\documentclass[a4paper,12pt]{article}

\usepackage{cmap}

\usepackage{graphicx}
\usepackage{subfigure}
\usepackage{grffile}
\usepackage{etoolbox}
\usepackage{amsmath,amsthm,amssymb}
\usepackage[T2A]{fontenc}
\usepackage[english, russian]{babel}
\usepackage[cp1251]{inputenc}
\usepackage{amsfonts,amssymb,mathrsfs,amscd}

\usepackage{hyperref} 

\usepackage{xcolor} 

\usepackage{comment} 

  \newskip\prethm \prethm3.0pt plus1.3pt minus.4pt
  \newskip\posthm \posthm2.7pt plus1.4pt minus.3pt
  \newtheoremstyle{STATEMENT}%
       {\prethm}{\posthm}{\itshape}{\parindent}{\scshape}
       {.}{.6em plus.2em minus.1em}{}
  \newtheoremstyle{EXPLANATION}%
       {\prethm}{\posthm}{}{\parindent}{\scshape}
       {.}{.6em plus.2em minus.1em}{}
\theoremstyle{STATEMENT}

\newtheorem{theorem}{Теорема}[section]

\newtheorem{lemma}{Лемма}[section]
\newtheorem{assertion}{Утверждение}[section]
\newtheorem{corollary}{Следствие}[section]

\theoremstyle{EXPLANATION}
\newtheorem{definition}{Определение}[section]
\newtheorem{remark}{Замечание}[section]

\numberwithin{equation}{subsection}

\title{Топологическая классификация слоений Лиувилля для интегрируемого
случая Ковалевской на алгебре Ли so(4)}
\author{В.\,А.~Кибкало}
\date{03.04.2018}

\begin{document}

\maketitle

Работа выполнена при поддержке Российского научного фонда, (грант № 17-11-01303).


\section{Введение}\label{s0}

С.В.~Ковалевской в знаменитой работе \cite{Kowalewski1889Acta12} был открыт новый случай интегрируемости уравнений динамики тяжелого твердого тела при всех значениях постоянной площадей. Данная система, рассмотренная на двойственном пространстве к алгебре Ли $e(3)$, была вложена И.В.~Комаровым \cite{Komarov81} в семейство динамических систем на пучке $so(3,1)-e(3)-so(4)$ с вещественным параметром $\varkappa$. Значению  $\varkappa =0$ соответствуют алгебра Ли $e(3)$ и интегрируемая система, называемая \textit{классическим случаем Ковалевской}. В работе исследуется случай алгебры Ли $so(4)$, который реализуется при $\varkappa > 0$. Соответствующую интегрируемую систему будем называть  \textit{компактным случаем Ковалевской}, т.к. для алгебры Ли $so(4)$ совместные уровни функций Казимира компактны. Особое множество и его топология изучались в работе И.К.~Козлова \cite{Kozlov14}

Для исследований классического случая Ковалевской использовались разнообразные методы. Например, Г.Г. Аппельротом в работе \cite{Appelrot40} были классифицированы особые орбиты с алгебраической точки зрения. М.П.~Харламовым в работах  \cite{Harlamov83Dan}, \cite{Harlamov83App}, \cite{Harlamov88} были развиты новые подходы, основанные на идеях Смейла исследования фазовой топологии и конструкциях симплектической геометрии.

Интересен ответ на вопрос, как устроены замыкания решений общего положения на неособых уровнях энергии для интегрируемых систем, открытых в физике и механике. Оказывается, совпадение таких замыканий решений двух систем задается отношением их лиувиллевой эквивалентности.

В работах А.Т.~Фоменко и его школы изучаются слоения Лиувилля вполне интегрируемых систем  на неособых трехмерных многообразиях $Q^3$ (см. работы \cite{KuNF08}, \cite{FKo12}, \cite{KuF12}, \cite{FKo14}, \cite{FNi15}, \cite{FF15}, \cite{VF17}). Эти поверхности расслоены на неособые торы и особые слои уровня дополнительного интеграла $F$. Две системы $M^4_i, v_i = \mathrm{sgrad}\, H_i$ на неособых $Q^3_i, i=1, 2$ называют \textit{лиувиллево эквивалентными}, если существует послойный диффеоморфизм $Q^3_1$ на $Q^3_2$, который сохраняет ориентации особых окружностей интеграла $F$. Теория топологической классификации слоений Лиувилля и ее приложения подробно изложены в книгах А.В.~Болсинова и А.Т.~Фоменко \cite{BF99}.

Классы лиувиллевой эквивалентности различаются построенным в работе \cite{FZ90} инвариантом Фоменко--Цишанга (меченой молекулой). Это конечный граф, оснащенный буквами и некоторыми числовыми метками $r, \varepsilon, n$. Его ребра соответствуют семействам регулярных торов. Вершины соответствуют особым слоям и обозначены символами 3-атомов (типов перестроек торов Лиувилля в окрестности этого особого слоя). Инвариант Фоменко--Цишанга без числовых меток, называемый графом Фоменко или грубой молекулой, также является важным инвариантом топологии слоения Лиувилля.

\begin{theorem}[$\left(\textrm{А.Т.\,Фоменко, Х.\,Цишанг}\right)$]\label{T:Fomenko_Zieschang}
 Слоения Лиувилля на неособых $Q^3$ лиувиллево эквивалентны, в точности если совпадают их инварианты Фоменко--Цишанга.
\end{theorem}
\begin{remark}
Таким образом, две нерезонансные системы имеют одинаковые замыкания решений общего положения, в точности если совпадают инварианты Фоменко--Цишанга их слоений Лиувилля.
\end{remark}

Перечисление инвариантов Фоменко--Цишанга системы и их сопоставление неособым уровням энергии называют \textit{лиувиллевым анализом}.

В работе \cite{BFR00}\, А.В.~Болсиновым, П.~Рихтером и А.Т.~Фоменко лиувиллев анализ был проведен для классического случая Ковалевской. В этой системе имеется ровно $10$ неэквивалентных слоений.

\begin{theorem}\label{T:invariants}
В системе Ковалевской на алгебре Ли $so\mathrm{(4)}$ имеется ровно $27$ классов $L_1, ..., L_{27}$ лиувиллево неэквивалентных слоений на связных компонентах неособых изоэнергетических поверхностей.

В таблицах \ref{Tab:Molecules_classes_labeled_1} и \ref{Tab:Molecules_classes_labeled_2} перечислены инварианты Фоменко--Цишанга всех слоений на неособых $Q^3_{a, b, h}$ при фиксированной ориентации $Q^3$ и направлении роста интеграла $K$. Классы лиувиллевой эквивалентности $L_1, \dots, L_{27}$ для связных компонент этих слоений $1$-$32$ указаны в таблице \ref{Tab:Liouville_classes_graphs}.
\end{theorem}

Пунктирная линия в таблицах \ref{Tab:Molecules_classes_labeled_1} и \ref{Tab:Molecules_classes_labeled_2} означает совпадение меток слева и справа от нее, т.е. симметрию слоения Лиувилля и его меченой молекулы на данных ``ярусах'', т.е. особых уровнях дополнительного интеграла.

\begin{theorem}\label{Cor:Liouville_equiv}
Среди слоений $L_1, \dots, L_{27}$ компактного случая Ковалевской следующие слоения лиувиллево эквивалентны слоениям известных интегрируемых систем в некоторых зонах энергии: 

1) слоения $L_1, L_{12}, L_3, L_4, L_{15}, L_{27}, L_{24}, L_{20}, L_{24}, L_{18}$ эквиваленты слоениям $A,\, ...\, ,J$ классического случая Ковалевской (см. \textrm{(\cite{BFR00})}) соответственно,

2) слоения $L_{2}$ и $L_{23}$ эквивалентны слоениям случая Ковалевской--Яхьи  в зонах энергии $h_{2}$ и $h_{10}, h_{23}$ (см.  \cite{Morozov07}, \cite{Slavina14}) соответственно,

3) слоения $L_1, L_2, L_9, L_{10}$ эквивалентны слоениям $1, 2, 6, 7$ случая Клебша (см. \cite{Morozov02}) соответственно,

4) слоения $L_1, L_2, L_4$ эквивалентны слоениям $A, B, F$ случая Соколова (см. \cite{Morozov04}) соответственно,


5) интегрируемые биллиарды в областях $A_0', A_2, A_1, A_0$, ограниченных дугами софокусных квадрик (см. \cite{Fokicheva15}), моделируют слоения Лиувилля $L_1, L_2, L_6, L_8$ компактного случая Ковалевской.
\end{theorem}

\begin{figure}[!htb]
\minipage{1.1\textwidth}
\includegraphics[width=\linewidth]{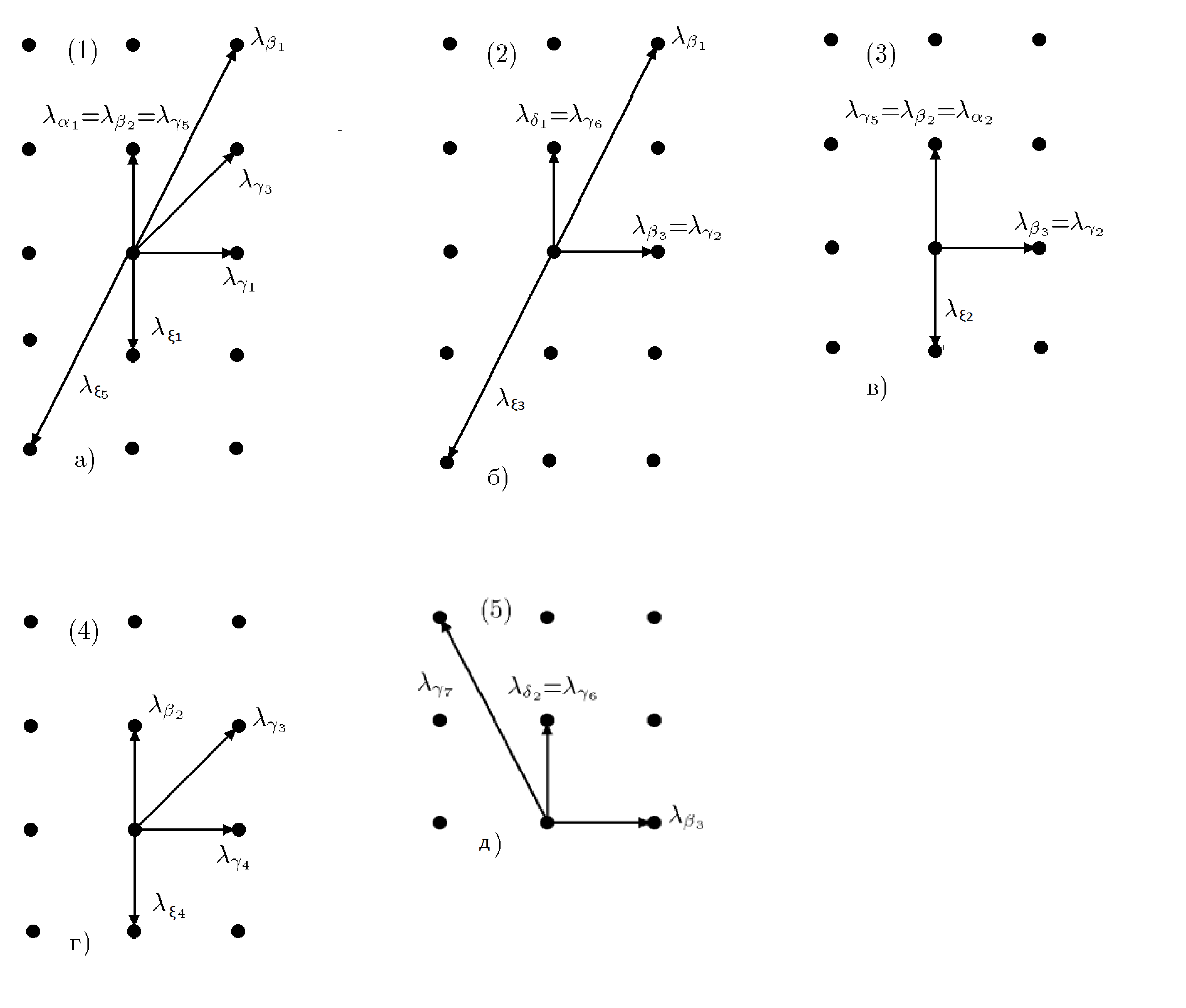}
   \caption{Циклы допустимых базисов для компактного случая Ковалевской} \label{Fig:Lambda_cycles}
\endminipage
\end{figure}

При проведении лиувиллева анализа мы используем результаты из работ \cite{BFR00} и \cite{Kozlov14}. В \cite{BFR00} для классического случая Ковалевской были найдены допустимые базисы для дуг бифуркационных диаграмм и вычислен инвариант Фоменко--Цишанга для слоений изоэнергетических поверхностей. В работе \cite{Kozlov14} для компактного случая Ковалевской были описаны все бифуркационные диаграммы отображения момента.

Таким образом, для нахождения всех меченых молекул компактного случая Ковалевской остается сделать следующее:

\begin{enumerate}

\item выразить через $\lambda$-циклы, известные из \cite{BFR00}, допустимые базисы для дуг $\xi_1, \dots, \xi_5$ (см. таблицу \ref{Table_arcs}) бифуркационных диаграмм, которые не возникали в классическом случае Ковалевской. Это сделано в утверждении~\ref{A:admiss_coord_systems}. На рисунке \ref{Fig:Lambda_cycles} однозначно определенные $\lambda$-циклы для дуг $\alpha_1, \dots, \xi_5$ изображены как элементы целочисленной решетки на плоскости.

\item описать, в каком порядке вдоль оси $H$ могут располагаться особые точки бифуркационных диаграмм в зависимости от параметров --- значений функций Казимира. Имеется лишь конечное число вариантов, все они приведены в таблицах \ref{Tab:order_region_b0}-\ref{Tab:order_region_I}. Интервалу между соседними особыми точками сопоставлен номер молекулы. Отметим, что в таблице \ref{Tab:order_region_II} для краткости не указано ``тривиальное окончание'' $6, z_2, 7, z_1$.
\end{enumerate}

Для дуг, сохраняющихся в классическом случае Ковалевской, мы используем те же обозначения $\alpha_1, \dots, \delta_2$ и те же базисы, что и в \cite{BFR00}. Для особых точек мы используем те же обозначения $y_i, z_j$, что и в \cite{Kozlov14}. Области I-IX плоскости значений функций Казимира, описанные в \cite{Kozlov14}, изображены на рисунках \ref{Fig:Areas_Big}-\ref{Fig:Areas_Small}. Полученное в настоящей работе их разбиение на подобласти изображено на рисунках \ref{Fig:Plane_Sep_uv_big}-\ref{Fig:sep_set_small}. Дуги, разделяющие подобласти, описаны в утверждении~\ref{A:Sep_set_1_arcs}.

Автор благодарит Анатолия Тимофеевича Фоменко и Ивана Константиновича Козлова за постановку задачи и помощь при написании работы.

\section{Постановка задачи} \label{s1}

\subsection{Динамические системы на пучке алгебр Ли} И.В.~Комаров в работе \cite{Komarov81} обнаружил, что на пучке алгебр Ли $so(4)-e(3)-so(3, 1)$ с вещественным параметром $\varkappa$ можно задать семейство интегрируемых систем с тем же параметром $\varkappa$, содержащее при $\varkappa =0$ классический случай Ковалевской, т.е. известный случай интегрируемости уравнений движения тяжелого твердого тела. Приведем ниже описание предложенной им конструкции.

Скобку Пуассона $\{ \cdot  , \cdot \} $ функций $f$ и $g$ на двойственном пространстве $\mathfrak{g}^*$ к конечномерной вещественной алгебре Ли $\mathfrak{g}$ можно задать следующим образом:
\begin{equation}
з\label{Eq:Lie_Poisson_Bracket} \{ f, g\} = \langle x, [df|_x, \,
dg|_x ]\rangle.
\end{equation} Запись $\langle \cdot  , \cdot \rangle $ обозначает спаривание ковектора из $\mathfrak{g}^*$ и вектора из $\mathfrak{g}$. Коммутатор в алгебре Ли  $\mathfrak{g}$ записан как $[ \cdot, \cdot ]$. Отметим, что мы пользовались каноническим отождествлением $\mathfrak{g}^{**}$ и $\mathfrak{g}$.

Динамическая система на $\mathfrak{g}^*$ в линейных координатах $x_1, \dots, x_n$ задается гладкой функцией $H$ с помощью уравнений Эйлера: \begin{equation}
\label{Eq:Lie_Poisson_Bracket_1} \dot{x}_i = \{ x_i, H
\},\end{equation}

Зададим на пространстве $\mathbb{R}^6$ с координатами $(\mathbf{J}, \mathbf{x}) = (J_1, J_2, J_3, x_1, x_2, x_3)$ семейство скобок Ли-Пуассона c параметром $\varkappa \in \mathbb{R}$, определяющее пучок алгебр Ли $\textrm{so}(3,1) - \textrm{e}(3) - \textrm{so}(4)$:

\begin{equation} \label{Eq:SO4_Poisson_Lie_bracket_1} \{J_i, J_j\} =
\varepsilon_{ijk}J_k, \quad \{J_i, x_j\} = \varepsilon_{ijk}x_k,
\quad \{x_i, x_j\} = \varkappa \varepsilon_{ijk}J_k, \end{equation}
где $\varepsilon_{ijk}$ --- знак перестановки
$\{123\} \rightarrow \{ijk\}$. Случай $\varkappa>0$ соответствует алгебре Ли $\textrm{so}(4)$, случай
$\varkappa=0$
--- алгебре Ли $\textrm{e}(3)$, а $\varkappa<0$ --- алгебре Ли $\textrm{so}(3, 1)$. Случай $\varkappa <0$ в настоящей работе не рассматривается.

Функции Казимира скобки Пуассона \eqref{Eq:SO4_Poisson_Lie_bracket_1} имеют вид:
\begin{equation}\label{Eq:Kasimirs}  f_1 = (x_1^2 + x_2^2 + x_3^2) + \varkappa (J_1^2 +J_2^2 +J_3^2),
\qquad  f_2 = x_1 J_1 + x_2 J_2 +x_3 J_3.
\end{equation}

В случае $\varkappa >0$ при $a \ge 2\sqrt{\varkappa}|b|$ и в случае $\varkappa =0$ при $a >0$ совместная поверхность уровня
\begin{equation}\label{Eq:Common_Level_surface} M_{a, b} = \{ (\textbf{J}, \textbf{x})| \quad  f_1
(\textbf{J}, \textbf{x}) = a, \, f_2 (\textbf{J}, \textbf{x}) = b
\}
\end{equation}
функций Казимира является орбитой коприсоединенного представления и симплектическим листом скобки \eqref{Eq:SO4_Poisson_Lie_bracket_1}. При $\varkappa \ge 0$ регулярными являются орбиты $M_{a, b}$, для которых $a > 2\sqrt{\varkappa}|b|$. Они диффеоморфны прямому произведению двумерных сфер $S^2\times S^2$ для $\varkappa >0$ и кокасательному расслоению к двумерной сфере $T^{*}S^2$ для $\varkappa =0$. При $\varkappa >0$ орбита сингулярна, если $a = 2\sqrt{\varkappa}|b|$. Орбита $M_{0, 0}$ является точкой, а остальные сингулярные орбиты диффеоморфны $S^2$. Слоение на них описано в лемме \ref{L:singular_orbits}. Парам $a <2\sqrt{\varkappa}|b|$ не соответствует ни одной орбиты.

И.В. Комаровым было найдено возмущение $K$ известного интеграла Ковалевской, находящееся в инволюции с гамильтонианом $H$ при всех $\varkappa \in \mathbb{R}$:
\begin{equation}
    \label{Eq:Hamiltonian}
    H = J_1^2 + J_2^2 + 2J_3^2 + 2 c_1 x_1,
\end{equation}
\begin{equation}\label{Eq:First_Integral} K = (J_1^2 - J_2^2-2c_1 x_1 + \varkappa
c_1^2)^2 + (2J_1 J_2 - 2 c_1 x_2)^2,
\end{equation} где $c_1$ --- произвольная постоянная. Можно считать, что $c_1 = 1$ и $\varkappa \in \{-1, 0, 1\}$.

\subsection{Интегрируемые системы на орбитах}
При всех $\varkappa \in \mathbb{R}$ первые интегралы $H$ и $K$ функционально независимы и находятся в инволюции относительно скобки \eqref{Eq:SO4_Poisson_Lie_bracket_1}. Т.е. на каждой регулярной орбите задана вполне интегрируемая по Лиувиллю гамильтонова система с двумя степенями свободы.

Отображение момента $\mathfrak{F} = (H, K): M^4_{a, b} \rightarrow \mathbb{R}^2(h, k)$ определяет на каждой орбите \textit{слоение Лиувилля}. Слоение на регулярной орбите удобно описывать с помощью бифуркационных диаграмм $\Sigma^{a, b}$ отображения $\mathfrak{F}$:

\[\Sigma^{a, b} = \{(h, k)| \,\,\exists x_0 \in \mathfrak{F}^{-1}(h, k), \,\,\, rk\, \mathfrak{F}|_{x_0} < 2 \}.\]

Бифуркационная диаграмма $\Sigma^{a, b}$ лежит в $\mathbb{R}^2(h, k)$ и состоит в случае $\varkappa \ge 0$ из конечного числа гладких дуг (возможно, неограниченных) и особых точек. Напомним, что \textit{особыми точками бифуркационной диаграммы} называют точки пересечения, касания и возврата или потери гладкости кривых, содержащих дуги $\Sigma^{a, b}$, т.е. концы этих дуг. 

\begin{remark}
Система имеет симметрию $(\textbf{J}, \textbf{x}) \rightarrow (-\textbf{J}, \textbf{x})$, переводящую точки совместного уровня $(a, b, h, k)$ функций $f_1, f_2, H, K$ в точки совместного уровня $(a, -b, h, k)$. Тем самым, бифуркационные диаграммы $\Sigma^{a, -b}$ и $\Sigma^{a, +b}$ состоят из одних и тех же точек плоскости $Ohk$, прообразы которых при отображении момента устроены одинаково. Далее считаем, что $b \ge 0$.
\end{remark}

Напомним, что  кривую $\gamma$ без самопересечений в плоскости $Ohk$ называют \textit{допустимой}, если она не проходит через особые точки $\Sigma^{a, b}$ и пересекает ее дуги трансверсально.

\textit{Изоэнергетическими поверхностями} $Q_{a, b, h}$ называют совместные поверхности уровня функций Казимира $f_1, f_2$ и гамильтониана $H$
\[Q_{a, b, h} = \{(\textbf{J}, \textbf{x})| \quad  f_1
(\textbf{J}, \textbf{x}) = a, \, f_2 (\textbf{J}, \textbf{x}) = b, \, H(\textbf{J}, \textbf{x}) = h\}.\]

Назовем $Q_{a, b, h}$ и соответствующую тройку $(a, b, h)$ \textit{неособыми}, если образ $\mathfrak{F}(Q_{a, b, h})$ является допустимой кривой. Отметим, что образ неособой $Q_{a, b, h}$ является отрезком, лежащим на вертикальной прямой $H =h$ плоскости $Ohk$, поскольку орбиты $M^4_{a, b}$ компактны в случае алгебры Ли $so(4)$. Неособая $Q_{a, b, h}$ является гладким трехмерным подмногообразием в $M^4_{a, b}$ без границы. Остальные непустые $Q_{a, b, h}$ будем называть \textit{особыми}.

\subsection{Результаты И.К.~Козлова для случая $so(4)$}

В работе И.К. Козлова \cite{Kozlov14} анализировался случай $\varkappa >0$. Были найдены четыре кривые, в объединении которых содержится бифуркационная диаграмма.

 \begin{lemma} \label{L:Curve_Images_B_not_0_Kappa_not_0}
Пусть $b \ne 0$ и $\varkappa \ne 0$. Тогда для любой регулярной орбиты
$M_{a, b}$ (такой, что $a^2 - 4 \varkappa
b^2 > 0$) бифуркационная диаграмма $\Sigma^{a, b}$ интегрируемой
системы с гамильтонианом \eqref{Eq:Hamiltonian} и
интегралом \eqref{Eq:First_Integral} содержится в объединении
следующих трёх семейств кривых на плоскости $\mathbb{R}^{2}(h,k)$:
  \begin{equation} \label{Eq:line_K_0}  \,\, 1. \quad \,\, \textrm{Прямая}  \qquad \qquad \qquad  \qquad \qquad
 k=0,  \qquad \qquad   \qquad \qquad \qquad \qquad \end{equation}
\begin{enumerate}
\item[2.] \quad Параметрическая кривая,
\begin{equation} \label{Eq:Parametric Curve} h(z)= \frac{b^2 c_1^2}{z^2}+2 z, \quad k(z)= \left(4 a
c_1^2-\frac{4 b^2 c_1^2}{z}+\frac{b^4 c_1^4}{z^4}\right)-2 \varkappa
c_1^2 h(z) +\varkappa^2 c_1^4, \end{equation} где $z \in \mathbb{R}
- \{0\}$.
\item[3.] \quad Объединение двух парабол \begin{equation} \label{Eq:Left Parabola} k =
\left(h-\varkappa c_1^2 -\frac{a}{\varkappa } +\frac{\sqrt{a^2-4
\varkappa b^2 }}{\varkappa }\right)^2 \end{equation} и
\begin{equation} \label{Eq:Right Parabola} k = \left(h-\varkappa c_1^2
-\frac{a}{\varkappa } - \frac{\sqrt{a^2-4 \varkappa b^2 }}{\varkappa
}\right)^2. \end{equation} \end{enumerate}
\end{lemma}

Взаимное расположение этих кривых, попадание их точек в образ отображения момента и устройство прообразов особых точек существенно зависит от значений $(a, b)$. И.К.~Козловым установлено, что $5$ кривых $f_l, f_t, f_k, f_r, f_m$ делят множество $a \ge 2\sqrt{\varkappa} b, b >0$ на $9$ областей с одинаковым устройством бифуркационной диаграммы:

\begin{theorem}[$\left(\textrm{И.К.\,Козлов, \cite{Kozlov14}}\right)$] \label{T:Bif_Diag}
Пусть $\varkappa > 0$ и $b > 0$. Функции $f_k, f_r, f_m, f_t$ и
$f_l$ заданные формулами
\begin{gather}
\label{Eq: Function_F_K} f_k(b) = \frac{3 b^{4/3}+6 \varkappa
b^{2/3} c_1^{4/3} - \varkappa ^2c_1^{8/3} }{4 c_1^{2/3}} \\
\label{Eq:Function_F_r} f_r(b) = \frac{b^{4/3}}{c_1^{2/3}}+
\varkappa b^{2/3}c_1^{2/3}
 \\ \label{Eq:Function_F_m} f_m(b)= \frac{b^2}{\varkappa c_1^2} +
 \varkappa^2c_1^2 \\ \label{Eq:Triple_Intersect} f_t(b) = \left(\frac{\varkappa c_1^2 + t^2}{2 c_1} \right)^2 + \varkappa t^2,
\qquad \text{где} \quad b =t \left(\frac{\varkappa c_1^2 + t^2}{2
c_1} \right) \\ \label{Eq:Function_F_l} f_l (b) = 2 \sqrt{\varkappa}
|b|
\end{gather} делят множество орбит $\{b>0, a \ge 2 \sqrt{\varkappa} b\} \subset
\mathbb{R}^2(a,b)$ на $9$ областей (см. рис. \ref{Fig:Areas_Big} и
\ref{Fig:Areas_Small}).
\end{theorem}

\begin{figure}[!htb]
\minipage{0.48\textwidth}
\includegraphics[width=\linewidth]{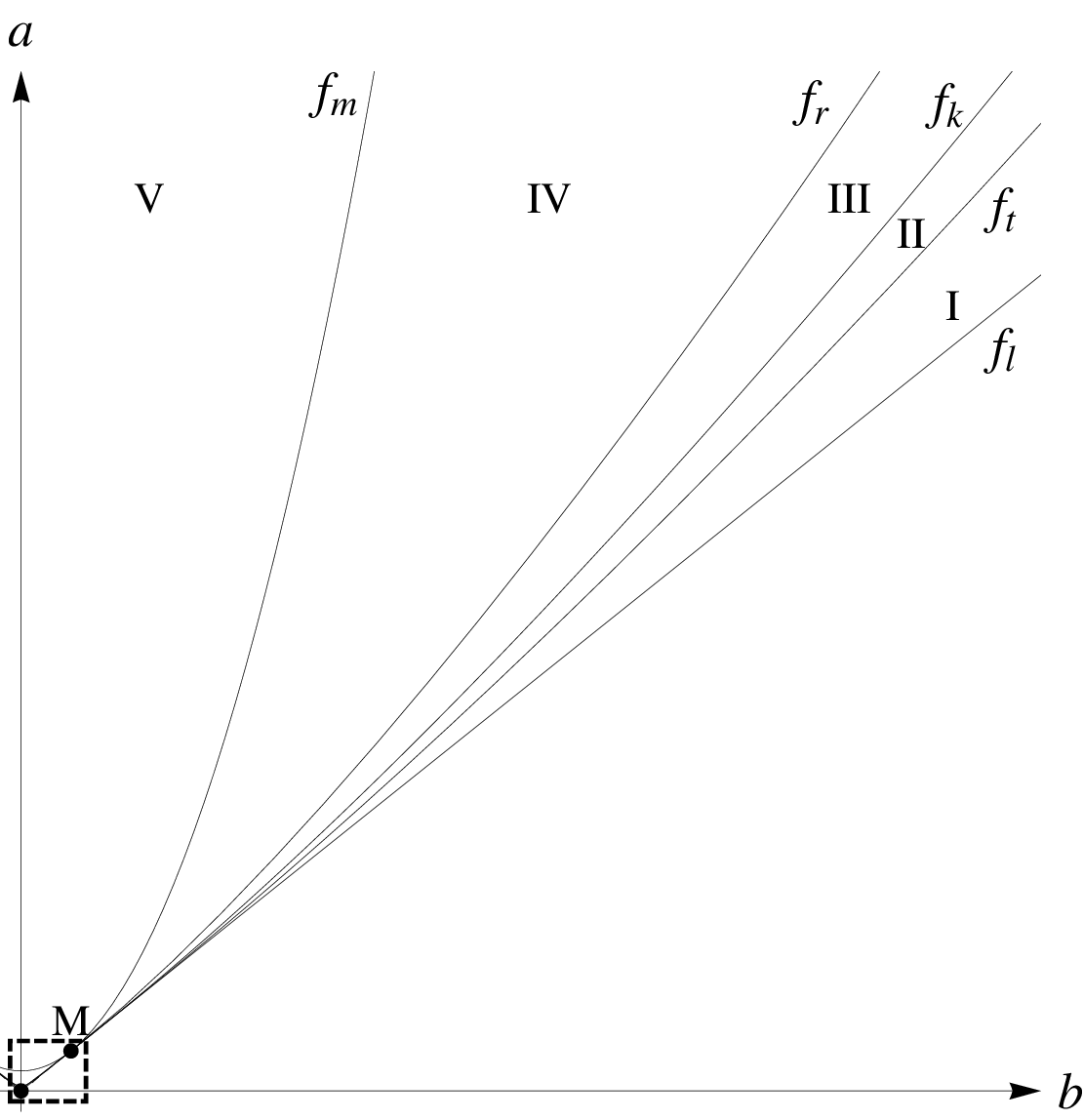}
   \caption{Разбиение области параметров} \label{Fig:Areas_Big}
\endminipage
\hspace{0.04\textwidth}
\minipage{0.48\textwidth}
\includegraphics[width=\linewidth]{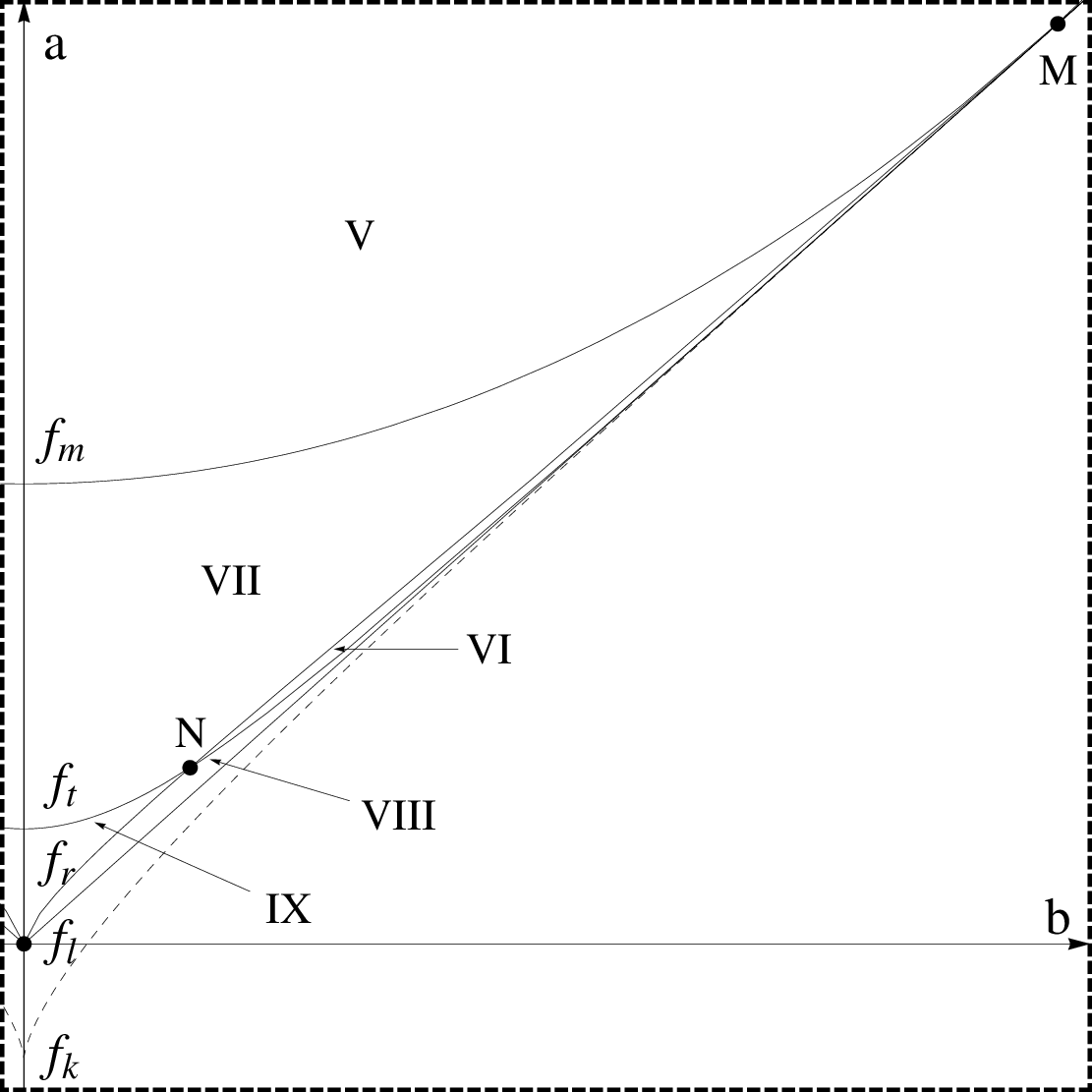}
   \caption{Увеличенный фрагмент рис. \ref{Fig:Areas_Big}.} \label{Fig:Areas_Small}
\endminipage
\end{figure}

Аналогичные утверждения были доказаны для луча $b =0,\, a >0$. Кроме этого, были построены круговые молекулы особых точек $y_1, ..., z_{11}$, проверена невырожденность критических точек отображения момента и установлена ``непрерывность''\, предельного перехода бифуркационных диаграмм $\Sigma^{a, b}(\varkappa)$ при $\varkappa \rightarrow +0$.

\begin{lemma} \label{L:Bif_Diag_Limit} Рассмотрим произвольные $a, b \in \mathbb{R}$, где
$a>0$.  Тогда точка $x$ принадлежит бифуркационной диаграмме $\Sigma
(a, b, 0)$ тогда и только тогда, когда существует последовательность
точек $x_n \in \Sigma(a_n, b_n, \varkappa_n)$ такая, что $\lim_{n
\to \infty} (a_n, b_n, \varkappa_n) = (a, b, 0)$.
\end{lemma}

Как следствие, из областей I-IX случая $\varkappa >0$ при предельном переходе ``выживают'' области I-IV, а из четырех кривых, содержащих $\Sigma^{a, b}(\varkappa)$ --- все кроме правой параболы \eqref{Eq:Right Parabola}.
При таком предельном переходе сохраняются круговая молекула и тип особой точки, если она не вырождается.
В таблице \ref{Tab:denotions_singular_points} показано соответствие обозначений для семейств особых точек в
 работах \cite{BFR00} и \cite{Kozlov14}. Отметим, что точка типа $y_4$, встречающаяся в $\Sigma^{a, b}$ при $b =0$, соответствует суперсингулярной точке $S_0$. Точка $z_7$ тоже присутствует только на орбитах с $b =0$, т.е. может быть названа суперсингулярной.

\section{Основные результаты} \label{s:Results}

\subsection{Выбор допустимых базисов}\label{s:Results_adm}

Для вычисления меток найдем допустимые базисы на граничных торах всех 3-атомов и вычислим матрицы перехода между ними, называемые  матрицами склейки. Общие правила выбора допустимых базисов изложены в \cite{BF99}. Для эллиптических атомов мы приведем эквивалентное описание перед утверждением \ref{A:Centre_Centre_Eps}.

По аналогии с работой \cite{BFR00} мы одновременно выберем допустимые базисы для всех перестроек и выразим их через однозначно определенные $\lambda$-циклы.  Для этого для каждой дуги бифуркационной диаграммы мы рассмотрим маленький вертикальный интервал $I$. Отметим, что в компактном случае Ковалевской у дуг бифуркационной диаграммы отсутствуют вертикальные касательные. После этого выберем допустимый базис для соответствующего 3-атома по правилам, сформулированным в \cite{BF99}.

Семействам дуг, для которых не существует аналогов в классическом случае Ковалевской, соответствует эллиптический атом $A$. Он диффеоморфен полноторию $S^1 \times D^2$, расслоенному на торы Лиувилля и одну особую минимальную или максимальную окружность. Базой такого расслоения является 2-атом А.
 Утверждение \ref{A:Centre_Centre_Eps} позволяет выбрать допустимые базисы на двух дугах, пересекающихся в особой точке, прообраз $\mathfrak{F}^{-1}$ которой содержит особенность типа центр-центр. Переформулируем ниже определение допустимого базиса для атома $A$ в более удобных для подсчета терминах.

\begin{definition}
Базис $(\lambda, \mu)$ в $\pi_1 (T^2)$ на граничном торе 3-атома А назовем допустимым, если
\begin{enumerate}
\item цикл $\lambda$ стягиваем,
\item ориентация цикла $\mu$ задана векторным
полем $\mathrm{sgrad}\, H$ на особом слое,%
\item базис $(\lambda, \mu)$ в $\pi_1(T^2)$ задает положительную ориентацию на граничном торе.
\end{enumerate}
 Базис $(u, v)$ в касательном пространстве $T_x T^2$ к граничному тору положительно ориентирован, если четверка векторов $(\textrm{grad}\,H, N, u, v)$ положительно ориентирована относительно формы объема $\omega \wedge \omega$. Здесь $N$ --- внешняя нормаль к $3$-атому. \end{definition}

Каждому фрагменту допустимой кривой, соединяющему концы двух вертикальных интервалов $I_i$ и не пересекающему дуги $\Sigma$, соответствует $2 \times 2$ матрица перехода от допустимого базиса на границе одного атома к допустимому базису на границе другого. Для изоэнергетической поверхности $H = \mathrm{const}$ определители таких матриц равны $-1$, а для поверхности $K = \mathrm{const}$ они равны $+1$. Можно доказать следующее общее утверждение о допустимых базисах для особенностей типа центр-центр.

\begin{assertion}\label{A:Centre_Centre_Eps}

Пусть точка является особой точкой типа центр-центр бифуркационной диаграммы. Обозначим знаки производных функции $H$ в направлении пересекающихся дуг $\gamma_i, i=1,2$ как $\varepsilon_i =\pm 1, i= 1,2$ соответственно. Тогда допустимые базисы $(\lambda_i, \mu_i)$ для этих дуг могут быть выбраны следующим образом:

\[  \begin{pmatrix} \lambda_2 \\ \mu_2 \end{pmatrix} =
    \begin{pmatrix} 0 & \varepsilon_1 \\ \varepsilon_2 & 0 \end{pmatrix}
    \begin{pmatrix} \lambda_1 \\ \mu_1  \end{pmatrix}
\]

\end{assertion}

По аналогии с работой \cite{BFR00} обозначим ``новые''\, дуги бифуркационной диаграммы символами $\xi_1, ..., \xi_5$. В таблице \ref{Table_arcs} эти дуги заданы своими концами $z_1, ..., z_{11}$. Семейства торов были пронумерованы в \cite{BFR00}.

\begin{assertion} \label{A:admiss_coord_systems}
 Следующие базисы $(\lambda_{\xi_i}, \mu_{\xi_i})$ являются допустимыми для дуг $\xi_i, i =1..5$:

  $\begin{pmatrix} \lambda_{\xi_1} \\ \mu_{\xi_1} \end{pmatrix} =
    \begin{pmatrix} \,\,1 & -1 \\ \,\,1 & \,\,\,\,0\,\,  \end{pmatrix}
    \begin{pmatrix} \,\,\lambda_{\gamma_1} \\ \,\,\lambda_{\gamma_3} \end{pmatrix}
  \qquad \begin{pmatrix} \lambda_{\xi_2} \\ \mu_{\xi_2} \end{pmatrix} =
    \begin{pmatrix} \,\,-1 & \,\,\,\,0\,\, \\ \,\,0 & \,\,\,\,1\,\,  \end{pmatrix}
    \begin{pmatrix} \,\,\lambda_{\beta_2} \\ \lambda_{\gamma_2} \end{pmatrix}
    $

 $\begin{pmatrix} \lambda_{\xi_3} \\ \mu_{\xi_3} \end{pmatrix} =
    \begin{pmatrix} \,\,0 & -1 \\ -1 & \,\,0  \end{pmatrix}
    \begin{pmatrix} \,\,\lambda_{\delta_1} \\ \,\,\lambda_{\beta_1} \end{pmatrix}
  \qquad %
  \begin{pmatrix} \lambda_{\xi_4} \\ \mu_{\xi_4} \end{pmatrix} =
    \begin{pmatrix} \,\,1 & -1 \\ \,\,1 & \,\,\,\,0\,\,  \end{pmatrix}
    \begin{pmatrix} \,\,\lambda_{\gamma_4} \\ \,\,\lambda_{\gamma_3} \end{pmatrix}
    $

$\begin{pmatrix} \lambda_{\xi_5} \\ \mu_{\xi_5} \end{pmatrix} =
    \begin{pmatrix} \,\,0 & -1 \\ \,\,1 & \,\,\,\,0\,\,  \end{pmatrix}
    \begin{pmatrix} \,\,\lambda_{\gamma_3} \\ \,\,\lambda_{\beta_1} \end{pmatrix}
    $
\end{assertion}

\begin{proof}
 \begin{enumerate}
 \item Рассмотрим орбиту, для которой $(a, b)$ лежит в области V.  Особая точка $z_4$ является точкой типа центр-центр. В обозначениях утверждения \ref{A:Centre_Centre_Eps} знаки производных $H$ в направлении дуг $\gamma_1$ и $\xi_1$ равны $\varepsilon_1 = -1$ и $\varepsilon_2 = 1$ соответственно. Cогласно ему, $\lambda_{\xi_1} =  p \lambda_{\gamma_1} - \lambda_{\gamma_3}$. Особая точка $z_{5}$ является точкой типа центр-седло, поэтому $\lambda_{\xi_1} = \pm \lambda_{\beta_2} = \pm (\lambda_{\gamma_3} - \lambda_{\gamma_1})$. Т.е. $p = 1$, и допустимый базис для $\xi_1$ имеет вид 
\[ \left( \begin{matrix} \lambda_{\xi_1} \\ \mu_{\xi_1} \end{matrix} \right) = \left( \begin{matrix} 1 & -1 \\ 1 & 0 \end{matrix} \right) \left( \begin{matrix} \lambda_{\gamma_1} \\ \lambda_{\gamma_3} \end{matrix} \right). \]

Аналогично, при $b =0$ получаем допустимый базис для дуги $\xi_4$ из анализа точки $z_7$ .

 \item Особая точка $z_2$ является вырожденной особой точкой ранга 1, поэтому можно выбрать $\mu_{\xi_2} = \lambda_{\gamma_2}$. Особая точка $z_{5}$ является точкой типа центр-седло, поэтому
\[ \left( \begin{matrix} \lambda_{\xi_2} \\ \mu_{\xi_2} \end{matrix} \right) = \left( \begin{matrix}
 \sigma & 0 \\ p & \sigma \end{matrix} \right)  \left( \begin{matrix} \lambda_{\beta_2} \\ -\lambda_{\gamma_2} \end{matrix} \right). \]

Получили, что $p = 0$ и $\sigma = -1$.

 \item Рассмотрим орбиту, для которой $(a, b)$ лежит в области $VII$.  Особая точка $z_{5}$ является точкой типа центр-седло, поэтому $\lambda_{\xi_3} = \pm \lambda_{\beta_1}$. Особая точка $z_1$ является точкой типа центр-центр, $\varepsilon_1 = -1$ и $\varepsilon_2 = -1$ для дуг $\delta_1$ и $\xi_3$ соответственно.  Тогда

\[ \left( \begin{matrix} \lambda_{\xi_3} \\ \mu_{\xi_3} \end{matrix} \right) = \left( \begin{matrix} 0 & -1 \\ -1 & 0 \end{matrix} \right) \left( \begin{matrix} \lambda_{\delta_1} \\ \lambda_{\beta_1} \end{matrix} \right). \]

 \item Особая точка $z_{10}$ является точкой типа центр-седло, поэтому $\lambda_{\xi_5} = \pm \lambda_{\beta_1}$. Особая точка $z_8$ является образом вырожденной одномерной орбиты, поэтому можно взять $\mu_{\xi_5} = \lambda_{\gamma_3}$.

 Рассмотрим орбиту, для которой $(a, b)$ лежит в области $VIII$. Особая точка $z_9$ является особой точкой типа центр-центр, $\varepsilon_1 = -1$ и $\varepsilon_2 = -1$ для дуг $\alpha_1$ и $\xi_5$ соответственно.  Тогда
\[ \left( \begin{matrix} \lambda_{\xi_5} \\ \mu_{\xi_5} \end{matrix} \right) = \left( \begin{matrix} 0 & -1 \\ 1 & 0 \end{matrix} \right) \left( \begin{matrix} \lambda_{\gamma_3} \\ \lambda_{\beta_1} \end{matrix} \right). \]
\end{enumerate}
Утверждение \ref{A:admiss_coord_systems} доказано.
\end{proof}

Инвариант Фоменко--Цишанга слоения Лиувилля в прообразе любой фиксированной допустимой кривой алгоритмически вычисляется по найденным выражениям для $\lambda_{\xi_i}, \mu_{\xi_i}$ и соотношениям из \cite{BFR00}. При вычислении меток $n$ для семей в $Q^3_{a, b, h}$ мы ориентировали ребра молекулы единообразно по возрастанию $K$ на прямой $H  = \mathrm{const}$. Напомним, что замена ориентации на $Q^3$ или знака дополнительного интеграла $K$ не меняет класс лиувиллевой эквивалентности.

\subsection{Классы слоений на неособых $Q^3_{a, b, h}$}\label{s:Results_param}

Напомним, что точке $(a, b, h)$ из пространства параметров $\mathbb{R}^3(a, b, h)$ соответствуют изоэнергетическая поверхность $Q_{a, b, h}$ со структурой слоения.
Каждая регулярная орбита $M^4_{a, b}$ задается парой $(a, b)$, где $a > f_l(b) >0$.  Бифуркационная диаграмма $\Sigma^{a, b}$ содержит конечное число особых точек, имеющих на плоскости $Ohk$ абсциссы $h_i$. Каждому такому $h_i$ соответствует или одна особая точка, или несколько.

Образ $\mathfrak{F}(Q_{a, b, h})$ лежит на прямой $H = h$ плоскости $Ohk$. Пересечения этой прямой с $\Sigma^{a, b}$ являются образами бифуркационных слоев. Для неособой $Q^3_{a, b, h}$ они не совпадают с особыми точками $\Sigma^{a, b}$, т.е. лежат на дугах $\Sigma^{a, b}$ и соответствуют атомам графа Фоменко (молекулы без меток). Приписав каждому атому графа название  $\alpha_1, \dots, \xi_5$ семейства пересекаемой дуги $\Sigma^{a, b}$, получим граф, который будем называть \textit{графом Фоменко с именованными атомами}.

\begin{assertion}\label{A:named_graphs}
В компактном случае Ковалевской встречается ровно $32$ различных графа Фоменко с именованными атомами (таблицы  \ref{Tab:Molecules_matrix_1}-\ref{Tab:Molecules_matrix_2}) слоений на неособых $Q^3_{a, b, h}$.
\end{assertion}

\begin{remark}
Из данного результата сразу следует теорема \ref{T:invariants}: инвариант Фоменко--Цишанга однозначно вычисляется по по графу Фоменко с именованными атомами и допустимым базисам для дуг $\alpha_1, \dots, \xi_5$.

В таблицах \ref{Tab:Molecules_matrix_1}-\ref{Tab:Molecules_matrix_2} также указаны вычисленные по допустимым базисам матрицы склейки на ребрах графов $1$-$32$. Их классы лиувиллевой эквивалентности $L_1, \dots L_{27}$ указаны в таблице~\ref{Tab:Liouville_classes_graphs}.
\end{remark}

Пусть все особые точки $\Sigma^{a, b}$ принадлежат семействам $y_1, \dots, z_{11}$ и имеют попарно различные абсциссы $h_i$. Тогда слоение на $M^4_{a, b}$ можно закодировать, записывая по очереди, по возрастанию $h$, названия семейств особых точек и номера графов Фоменко с именованными атомами, соответствующие интервалам $(h_i, h_{i+1})$ оси $Oh$. В качестве примера приведем код для любой точки $(a, b)$ из области VIII:
\[\textrm{VIII:} \quad  y_1,\, 1,\, z_4,\, 13, z_{11},\, 11\,, z_9.\]

\begin{definition}\label{Def:sep_set}
Назовем \textit{разделяющим множеством} $\Theta$ множество пар $(a, b)$ плоскости $Oab$, в окрестности которых есть пары $(a_i, b_i)$ с различными кодами или пара $(a', b')$, которой не соответствует кода.
\end{definition}

Кривые $f_l, f_t, f_k, f_r, f_m$, описанные И.К.Козловым в \cite{Kozlov14}, содержатся в $\Theta$, т.к.\,их точкам соответствуют некоторые перестройки бифуркационных диаграмм. В частности, для точек кривой $f_l$ орбита $M_{a, b}$ сингулярна, а для точек $(a, b)$ остальных кривых диаграмма $\Sigma^{a, b}$ имеет вырожденную особую точку, не принадлежащую семействам $y_1, \dots, z_{11}$. Эти кривые разделяют множество орбит c $b >0$ на 9 областей I-IX, а луч $b = 0, a >0$ на три промежутка X-XII.

Остальные дуги $\Theta$ состоят из точек $(a, b)$, для которых совпадают значения $h$ гамильтониана $H$ для двух и более особых точек $\Sigma^{a, b}$. Следующее утверждение описывает, как разделяющее множество делит плоскость $Oab$.

\begin{assertion}\label{A:subregions}
В компактном случае Ковалевской

1) разделяющее множество $\Theta$, описанное в утверждении \ref{A:Sep_set_1_arcs} как конечное объединение дуг, изображено на рисунках \ref{Fig:Plane_Sep_uv_big}-\ref{Fig:sep_set_small} в координатах $(u, v)$.

2) $\Theta$ разбивает области I-IX плоскости $Oab$ на $61$ подобласть, а промежутки X-XII оси $Oa$ разбиваются на 4 интервала и луч. Точкам из одного подмножества соответствуют одинаковые коды, приведенные в таблицах \ref{Tab:order_region_small}-\ref{Tab:order_region_I} для случая $b \ne 0$ и в таблице \ref{Tab:order_region_b0} для оси $b =0$.

3) номер графа Фоменко с именованными атомами, соответствующего произвольному неособому уровню $h$ для пары $(a, b) \in \Theta \backslash f_l$, входит в код каждой из подобластей, содержащих эту пару $(a, b)$ в своей границе.
\end{assertion}

\begin{remark}
Для краткости в кодах подобластей области II опустим одинаковые символы в конце кода: $\dots, 6, z_2 ,7, z_1$.
\end{remark}

На рисунках указаны номера всех подобластей и функций $f_i$, задающих ребра $\Theta$. Вершины $\Theta$, в которых пересекается более двух кривых, пронумерованы от 0 до 9.

\begin{remark}
Отметим, что на рисунке \ref{Fig:Plane_Sep_uv_big}-\ref{Fig:sep_set_small} разделяющее множество изображено в других координатах $(u, v)$. Из леммы \ref{L:change_ab_uv} следует, что образ $\Theta$ при сделанной замене $(a, b)\rightarrow (u, v)$ удовлетворяет определению \ref{Def:sep_set}.
\end{remark}

\begin{figure}[!htb]
\minipage{0.95\textwidth}
\includegraphics[width=\linewidth]{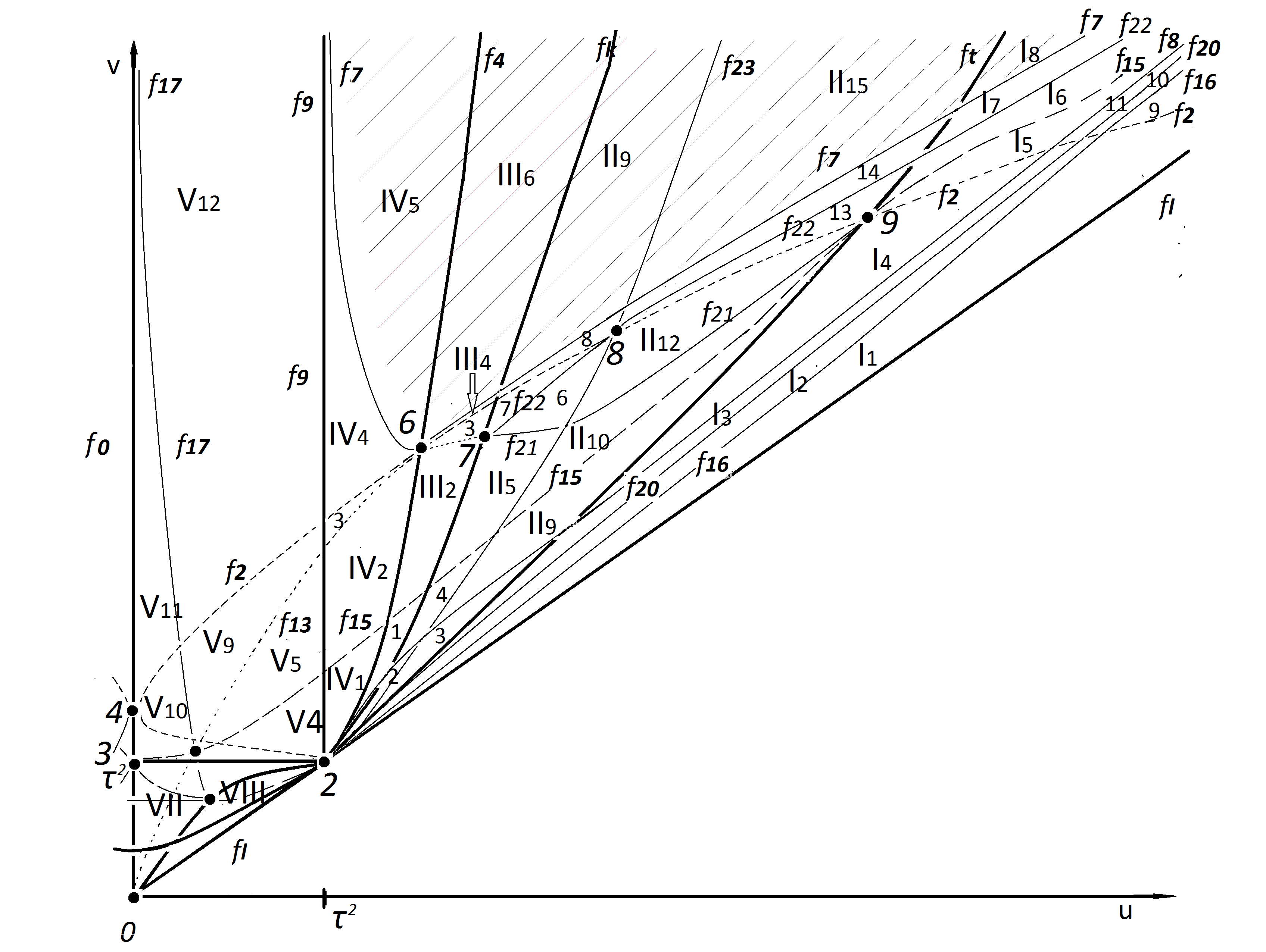}
   \caption{Разделяющее множество: области I-V} \label{Fig:Plane_Sep_uv_big}
\endminipage
\end{figure}

\begin{figure}[!htb]
\minipage{0.80\textwidth}
\includegraphics[width=\linewidth]{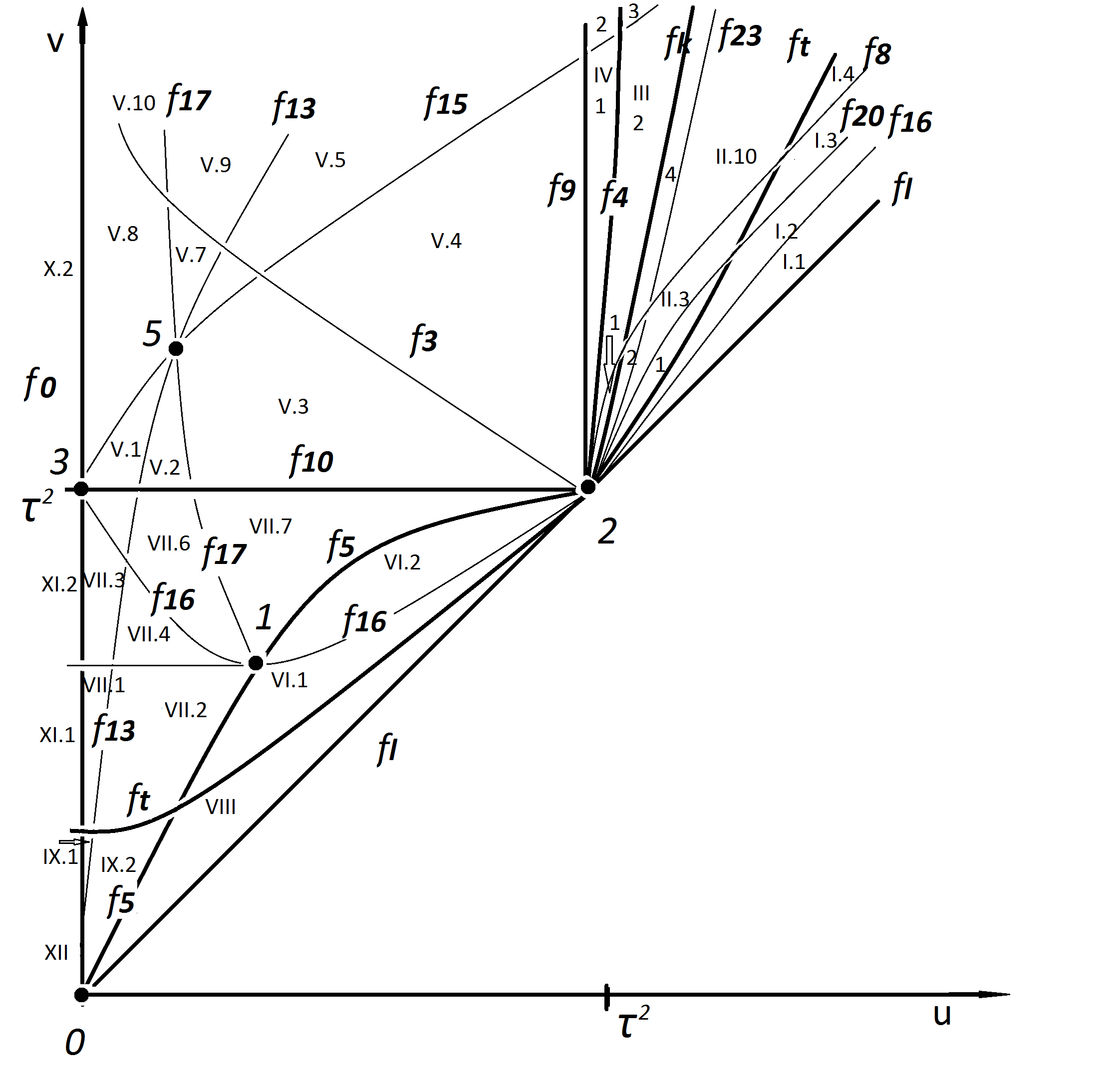}
   \caption{Увеличенный фрагмент рисунка \ref{Fig:Plane_Sep_uv_big}} \label{Fig:sep_set_small}
\endminipage
\end{figure}

\subsection{Стратификация пространства параметров системы}

\subsubsection{Трехмерные камеры неособых троек $(a, b, h)$}
Опишем, каким тройкам $(a, b, h)$ пространства параметров $\mathbb{R}^3(a, b, h)$ соответствуют одинаковые графы Фоменко с именованными атомами. Назовем \textit{камерой} связное множество троек $(a, b, h)$ с неособой 
$Q_{a, b, h}$, а \textit{особым множеством} $\mathbb{A}^2$ --- множество троек $(a, b, h)$ с особой $Q_{a, b, h}$.

Заметим, что если особая точка $\Sigma^{a, b}$ принадлежит семействам $y_1, \dots, \widehat{y_4}, \dots,$ $\widehat{z_7}, \dots, z_{11}$, то ее абсцисса $h$ является гладкой функцией от значений функций Казимира. Для нулей параметрической кривой \eqref{Eq:Parametric Curve} это верно, т.к. функции $h(a, b, z)$ и $k(a, b, z)$, определяющие ее, являются полиномами от $z$ и $z^{-1}$. Для оставшихся точек это следует из леммы \ref{L: h_coord_points}, в которой описаны явные уравнения поверхностей. Отметим, что при $b =0$ некоторые из поверхностей первой серии могут склеиваться друг с другом непрерывно, но не гладко. Множеству точек склейки соответствуют особые точки типа $y_4$ и $z_7$. Из данных соображений следует, что объединение границ всех камер совпадает с особым множеством $\mathbb{A}^2$. Отсюда особое множество замкнуто, а все камеры открыты. Сформулируем основное утверждение об устройстве пространства параметров $\mathbb{R}^3(a, b, h)$.

\begin{assertion}\label{A:cameras}
Каждая камера в пространстве параметров $\mathbb{R}^3(a, b, h)$ является открытым диском $D^3$, всем точкам $(a, b, h)$ которого соответствует одинаковый граф Фоменко с именованными атомами слоения на $Q^3_{a, b, h}$.

Каждая камера либо симметрична самой себе относительно плоскости $b =0$ и трансверсально пересекает ее, либо не имеет с ней общих точек, расположена строго по одну сторону от нее и симметрична другой камере.

Каждый граф Фоменко с именованными атомами реализуется в одной камере (графы 1-11) или паре камер (графы 12-32), симметричных относительно плоскости $b =0$.
\end{assertion}

\begin{remark}
Множество троек $(a, b, h)$ с пустыми $Q_{a, b, h}$ также гомеоморфно трехмерному диску. 
\end{remark}

Тем самым, каждой камере в полупространстве $b >0$ можно биективно сопоставить граф Фоменко с именованными атомами. Далее для каждого из ребер $\Theta$ можно перечислить графы Фоменко с именованными атомами, соответствующие неособым тройкам $(a, b, h)$ с парой $(a, b)$, лежащей на данном ребре.

Односвязность камеры следует из односвязности слоя проекции $(a, b, h) \rightarrow (a, b)$ камеры на плоскость $Oab$ и односвязности базы. Она проверяется явно по таблицам кодов \ref{Tab:order_region_small}-\ref{Tab:order_region_I} и \ref{Tab:order_region_b0}. Пусть некоторый номер графа Фоменко c именованными атомами присутствует в кодах двух соседних по дуге из $\Theta$ подобластей и для $(a, b)$ с этой дуги, кроме возможно, ее концов. Из пункта 3 утверждения \ref{A:subregions} на их объединение проецируется одна камера с данным графом.

\subsubsection{Граф соседства трехмерных камер}

 Назовем \textit{соседними} две камеры, которые граничат друг с другом по двумерному подмножеству. Напомним, что граница каждой камеры лежит в  особом множестве $\mathbb{A}^2$.

 Максимальное связное подмножество $\mathbb{A}^2$, каждой точке $(a, b, h)$ которого соответствует ровно одна особая точка $\Sigma^{a, b}$ фиксированного семейства $y_1, \dots, z_{11}$ кроме $y_4$ и $z_{11}$, назовем \textit{гранью}. Проекция такой грани на плоскость $Oab$ является инъекцией.

Теперь определим грани, лежащие в плоскостях $a = f_l(b)$ пространства $\mathbb{R}^3(a, b, h)$. Парам $a = f_l(b) >0$ соответствуют сингулярные двумерные орбиты $M^2_{a, b} \cong S^2$. Для каждой из них множество троек $(a, b, h)$ c непустыми $Q_{a, b, h}$ является точкой или вертикальным отрезком.

 \begin{lemma} \label{L:singular_orbits}
Пусть $\varkappa >0$ и $a = f_l(b)$, тогда на $M^2_{a, b}$ функции $H$ и $K$ зависимы, неособый слой гомеоморфен одной или двум окружностям $S^1$ или $2S^1$. Слоение на $M^2_{a, b}$ можно однозначно описать его особыми слоями:
\begin{itemize}
\item при $0< b \le \sqrt{\varkappa^3 c_1^4}$ оба особых слоя, минимальный и максимальный, являются точкой,

 \item при $\sqrt{\varkappa^3 c_1^4} < b$ седловой слой гомеоморфен восьмерке, минимальный является точкой, а  максимальный --- двумя точками.
\end{itemize}
\end{lemma}

Паре $(0, 0)$ соответствует одноточечная орбита $M_{0, 0}$, а при $a= 2\varkappa^2 c_1^2, b = \varkappa^{3/2} c_1^{2}$ происходит перестройка особых слоев. Здесь \textit{гранями} назовем максимальные связные подмножества троек $(a, b, h)$ с фиксированным числом окружностей в $Q_{a, b, h}$, реализующих неособые слои в $M^2$. Следующее утверждение доказывается аккуратным рассмотрением проекций граней на плоскость $Oab$.

\begin{assertion}
Каждая грань является открытым двумерным диском. Пересечение границ двух соседних камер (за исключением пары камер 1  и 2), содержит ровно одну грань. В таблице \ref{Tab:polyhedron_facets} перечислены все пары соседних камер. Для каждой из них указано семейство особых точек, соответствующее точкам грани.
\end{assertion}

\begin{remark}
Пара камер 1 и 2 граничит по двум граням, симметричным относительно плоскости $b =0$. Их точкам соответствует семейство $y_6$. В плоскости $b =0$ точкам кривой, являющейся границей этих граней, соответствует особая точка из семейства $y_4$. При этом типичная особенность типа седло-узел перестроилась в типичную особенность типа ``эллиптическая вилка'', см. \cite{BFR00}.

Семейство $z_7$ соответствует точкам кривой в плоскости $b =0$, лежащей в границе нескольких граней. При $b \rightarrow 0$ не происходит перестроек прообразов особых точек типа $z_4$ и $z_6$, имеющих тип центр-центр.
\end{remark}

\begin{remark}
В таблице \ref{Tab:polyhedron_facets} камеры в паре упорядочены по возрастанию $h$. Для граней $p_{65}, p_{66}$ (и граней, симметричных им при $b <0$)  указано семейство особых точек $\Sigma^{a, b}$ или количество окружностей, одна или две, составляющих $Q_{a, b, h}$. Знак $\oslash$ обозначает множество всех троек $(a, b, h)$ с пустыми $Q_{a, b, h}$.
\end{remark}

\subsection{Связь с классическим случаем Ковалевской}\label{s:Results_classcase}%

Многие динамические системы, возникающие в механике, могут рассматриваться как системы на алгебре Ли $e(3)$. В классическом случае Ковалевской можно без потери общности положить геометрический интеграл $f_1$ равным $1$. Тогда плоскость $Obh$ разбивается несколькими кривыми на области с лиувиллево эквивалентными слоениями $Q^3_{1, b, h}$. На рисунке \ref{Fig:Projection_h_l} изображено данное разбиение, т.е. сечение особого множества $\mathbb{A}^2$ классического случая Ковалевской уровнем $a =1$. Сплошные кривые разделяют области с негомеоморфными $Q^3_{1, b, h}$, а пунктирные линии --- области с гомеоморфными $Q^3_{a, b, h}$, но неэквивалентными слоениями Лиувилля.  Им соответствуют типы $e_1, e_2, c_1, c_2, h_1, h_2$ особых точек --- вырождения классов Аппельрота \cite{BFR00}.

В случае $\varkappa =0$ прообразы особых точек бифуркационных диаграмм $\Sigma^{1, b}$  содержатся в поверхностях $Q^3_{1, b, h}$, соответствующих изображенным кривым. Данные линии имеют $9$ общих точек. Семь из них соответствуют вырождению критических точек. В \cite{BFR00} они назывались суперсингулярными и были объединены в классы $S_0, S_1, S_2, S_3$. Отметим, что при пересечении кривых $M$ и $U$ вырождения не происходит --- образы критических точек на $Ohk$ имеют разное значение интеграла $K$.

\begin{figure}[!htb]
\begin{center}
\minipage{0.70\textwidth}
\includegraphics[width=\linewidth]{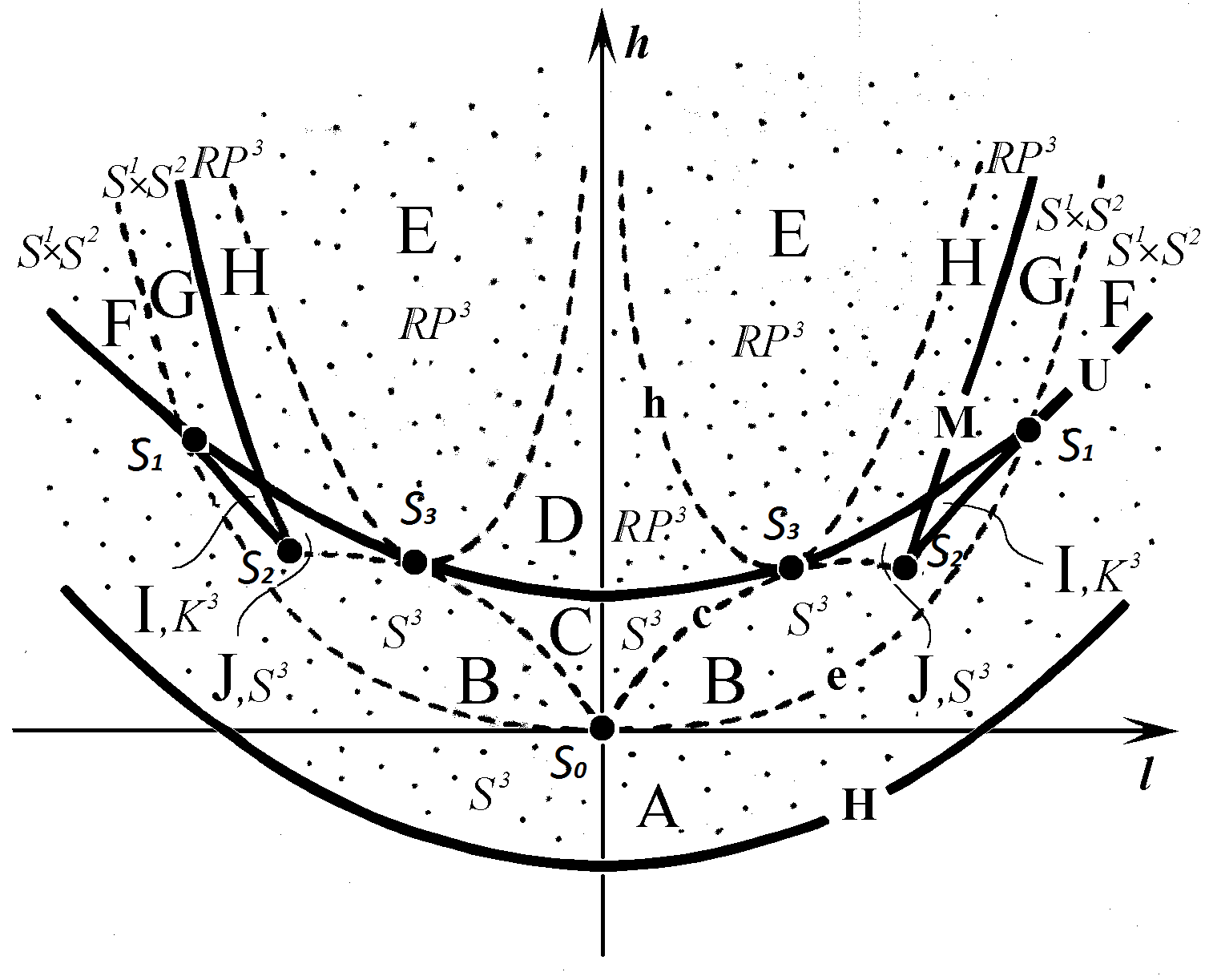}
   \caption{Проекция $\Sigma_{b, h}$ в случае $\varkappa =0$} \label{Fig:Projection_h_l}
\endminipage
\end{center}
\end{figure}

 В работе \cite{Kozlov14} установлено, что при предельном переходе $\varkappa \rightarrow +0$ сохраняются типы перестроек для ``старых'' дуг $\alpha_1, \dots, \delta_2$ и особых точек $y_1, \dots, y_{13}$.  Тем самым, области I-IV компактного случая Ковалевской соответствуют интервалам, на которые ось $Ob$ разбивается проекциями суперсингулярных точек на рисунке \ref{Fig:Projection_h_l}.  Отметим, что координата $b$ точки пересечения кривых $M$ и $U$ разделяет один из интервалов на два, с разным порядком особых точек $y_3$ и $y_{10}$ по оси $Oh$.

По следствию \ref{Cor:Liouville_equiv} каждое слоение на неособом $Q^3_{a, b, h}$ классического случая Ковалевской лиувиллево эквивалентно одному из слоений компактного случая.

\begin{assertion}\label{A:model_classcase}
 Заштрихованные на рисунке \ref{Fig:Plane_Sep_uv_big} подобласти I.8, II.15, II.9, III.6 и IV.5 компактного случая Ковалевской, 
 лежащие выше кривой $v = f_{7}(u) =  2 u^{3/2} / (\sqrt{u} - \sqrt{2}\varkappa c_1)$, соответствуют интервалам классического случая Ковалевской, а именно
 \begin{enumerate}
 \item в данных подобластях компактного случая все ``старые'' (имеющие аналоги в классическом случае Ковалевской) особые точки $y_i \in \Sigma^{a, b}$ лежат левее всех ``новых'' (не имеющих таких аналогов) особых точек $z_i$.

 \item начало кода (т.е. последовательность до первой точки $z_i$ особых точек $y_i$ и молекул Фоменко с именованными атомами при возрастании $h$) этих подобластей компактного случая совпадает с кодами интервалов классического случая Ковалевской.
 \end{enumerate}
 \end{assertion}

 Отметим, что соседним интервалам оси $Ob$ на рисунке \ref{Fig:Projection_h_l}, ограниченным значениями $b_i$ общих точек изображенных на нем кривых соответствуют соседние подобласти компактного случая. В силу независимости устройства сечения от выбора $a >0$ для случая $\varkappa =0$, можно говорить о ``вложении'' классического случая Ковалевской при ненулевой постоянной площадей в компактный случай Ковалевской.

\subsection{Устройство разделяющего множества}

\subsubsection{Поверхности в $\mathbb{R}^3(a, b, h)$ и особые точки $\Sigma^{a, b}$}

Абсциссы $h$ особых точек $\Sigma^{a, b}$ кроме нулей параметрической кривой \eqref{Eq:Parametric Curve} удовлетворяют хотя бы одному из уравнений $h = h_i(a, b)$ или $h = h(a, b, z_i(a, b))$. Все особые точки  из одного семейства $y_1, \dots, \widehat{y_4}, \dots, \widehat{z_7}, \dots, z_{11}$ удовлетворяют ровно одному из этих уравнений.

Графиками этих функций от $(a, b)$ являются двумерные поверхности в $\mathbb{R}^3(a, b, h)$, обозначим их $int, l$ и $\pm r, \pm l, lt, rt, cusp$. Будем называть их и соответствующие им особые точки $\Sigma^{a, b}$ \textit{поверхностями} и \textit{особыми точками первой серии}. В таблице~\ref{Tab:surface_family} указано соответствие поверхностей и семейств особых точек.

 В зависимости от $(a, b)$ точке $(a, b, h)$ такой поверхности может соответствовать одна или несколько особых точек, или $\Sigma^{a, b}$ может не содержать особых точек с такой абсциссой $h$.

Тройкам $(a, b, h)$ из поверхностей $\pm r, \pm l$ соответствуют точки пересечения параметрической кривой \eqref{Eq:Parametric Curve} с левой и правой параболами, тройкам из $ lt, rt$ --- касания этих кривых. Тройки из поверхности $cusp, int$ и $l$ соответствуют точке возврата кривой \eqref{Eq:Parametric Curve}, точке пересечения парабол и вершине левой параболы.

Нулям параметрической кривой \eqref{Eq:Parametric Curve}, в том числе особым точки из семейств $y_{10}, y_{11}$ и $z_1$, соответствует поверхность $root$, заданная неявно: $k(a, b, z) = 0$. Она не является однозначной над областями I и II. Эти семейства и поверхность будем относить ко \textit{второй серии}.

Для удобства дальнейших вычислений перепараметризуем множество орбит $D = \{(a, b)| a \ge f_l \ge 0\}$ значениями $(u, v)$ новых функций Казимира, которые выражаются через $f_1, f_2$ как решения уравнения $f^2 - 2 f_1 f + 4 f_2^2 = 0$.

\begin{equation}\label{Eq:uv_ab}
\begin{cases}
u = a - \sqrt{a^2- 4\varkappa b^2}&\\
v = a + \sqrt{a^2- 4\varkappa b^2}&\\
\end{cases},
\quad
\begin{cases}
a = (u+v)\slash 2 &\\
b^2 = uv \slash (4\varkappa) &\\
\end{cases}
\end{equation}

Следующая лемма о порождаемом при этом отображении $\rho: \, (a, b) \rightarrow (u, v)$ в плоскости $Oab$ доказывается явным вычислением. При этом образом лучей ${b= 0, a \ge 0}$ и ${a = f_l(b) \ge 0, b \ge 0}$ являются лучи ${u= 0, v \ge 0}$ и ${v= u, u \ge 0}$. Далее сохраним для $\rho(D)$ обозначение $D$.

\begin{lemma}\label{L:change_ab_uv}
Пусть $\varkappa >0$, тогда замена \eqref{Eq:uv_ab} регулярна в $\mathrm{Int}\,D$ и является гомеоморфизмом множеств $D$ и $\{(u, v)|\, v \ge u \geq 0\}$.
\end{lemma}

Функции $h_i(a, b)$ и $z_i(a, b)$ для поверхностей первой серии были найдены в работе \cite{Kozlov14}. В лемме \ref{L: h_coord_points} они будут записаны как функции от $(u, v)$.

\subsubsection{Равенство абсцисс особых точек первой серии} Укажем орбиты, для которых равны значения функций $h_i$ первой серии, т.е. совпадают абсциссы $h_i$ особых точек из семейств первой серии или происходит перестройка особых точек.

В утверждении \ref{A:Curves_sep_set_2} попарно приравняем их функции $h_i(a, b)$ и $h(a, b, z_i(a, b))$ и решим возникающие уравнения. Иными словами, поверхности первой серии будут в $\mathbb{R}^3(a, b, h)$ пересекаться или касаться над кривыми на плоскости орбит, перечисленными в таблицах \ref{Tab:sep_curves_2.1} и \ref{Tab:empty_curves}.

\begin{assertion} \label{A:Curves_sep_set_2}
 Пусть $\varkappa > 0$, тогда поверхности первой серии пересекаются в точности для пар $(u, v)$, лежащих на разделяющих кривых $f_0, ..., f_{19}$.

   В таблицах \ref{Tab:sep_curves_2.1} и \ref{Tab:empty_curves} для них указаны явные формулы вида $y = y(x)$ и значения $x_1, x_2$ аргумента $x$, при которых $f_i$ лежит в $D$.
\end{assertion}

\begin{remark}
В таблицах \ref{Tab:sep_curves_2.1} и \ref{Tab:empty_curves}
\begin{enumerate}
\item в указанных интервалах $x \in (x_1, x_2)$ функция $f_i$ определена, и точка с координатами $x$ и $f_i(x)$ лежит в $D$, т.е. ей соответствует непустая орбита.
\item семейства в паре упорядочены по возрастанию абсцисс $h$ их особых точек при $f_{i} (x) +\varepsilon > y > f_{i} (x)$, т.е. там $h_1 \le h_2$. Паре сопоставлен нижний индекс $0$, если на кривой достигается минимум $h_2 - h_1$, и $1$, если знак этой разности изменяется при переходе через кривую.
\item кривые $f_l$ и $f_t$ являются графиками функций $f_1$ и $f_{14}$, а кривые $f_r$, $f_m$ --- объединением графиков функций $f_4$ и $f_5$, $f_9$ и $f_{10}$ соответственно.
\end{enumerate}
\end{remark}

Утверждение \ref{A:Curves_sep_set_2} доказано в пункте \ref{s:Proof_proof_A:Curves_sep_set_2}.

\begin{corollary}
Пусть $\varkappa >0$, тогда луч  $a > 0, b = 0$ содержит ровно $4$ точки, для которых имеются пересечение поверхностей первой серии:
\[\cfrac{\varkappa^2 c_1^2}{4}:\, l, +r, z_1, \qquad \frac{\varkappa^2 c_1^2}{2}:\, l, rt, \qquad \varkappa^2 c_1^2:\, +l, int, rt,   \qquad  4 \varkappa^2 c_1^2:\, +l, -r\]
\end{corollary}

\subsubsection{Равенство абсцисс особых точек из разных серий}

 Теперь опишем разделяющие кривые, для точек $(u, v)$ которых нуль параметрической кривой имеет абсциссу $h$, удовлетворяющую одному из уравнений первой серии $h = h_i(u, v)$ или $h = h(u, v, z_i(u, v))$.
 Отметим, что этим нулем не является точка $z_1$ --- самая правая особая точка $\Sigma^{a, b}$, если содержится в ней.

 Первой такой кривой является $f_k$, на которой точка возврата с параметром $z_{cusp}$ является нулем кривой \eqref{Eq:Parametric Curve}, т.е. $k(f_k(b), b, z_{cusp}(f_k(b),b)) = 0$. Оставшиеся четыре кривые $f_{20}, f_{21}, f_{22}, f_{23}$ описаны в утверждении \ref{A:curves_sep_set_1}. Рассмотрим замены координат $q = q(u, v), s = s(v)$ и $\widetilde{q} = q(v, u), \widetilde{s} = s(u)$.

\begin{equation}\label{Eq:sq_coordinates}
\begin{cases}
q(u, v) = \cfrac{\sqrt{u\mathstrut}}{v}&\\
s(v) = 1/v&\\
\end{cases},
\qquad 
\begin{cases}
u = \cfrac{q^2}{s^2} &\\
v = 1/s&\\
\end{cases}
\end{equation}

\begin{lemma}\label{L:change_uv_qs}
Пусть $\varkappa >0$, тогда замена \eqref{Eq:sq_coordinates} задает гомеоморфизм множеств $E = \{(u, v)| v \ge u \ge \tau^2\}$ и  $\{(q, s)\,| \,q / \tau \ge s \ge q^2, 0 < q \le 1/\tau\}$ и регулярна в $\mathrm{Int}\,E$.
\end{lemma}

В координатах $(q, s)$ разделяющие кривые $f_l, f_t, f_r$ и $f_m$ являются графиками функций одной переменной, как и в исходных координатах $(a, b)$. Эти функции указаны в таблице \ref{Tab:sep_curves:q_s-type}. Для кривой $f_k$ верен аналогичный факт, который будет доказан в разделе \ref{s:Proof_proof_L:bif_curves_sq}.

\begin{lemma}\label{L:bif_curves_sq}
Пусть $\varkappa >0$, тогда в координатах \eqref{Eq:sq_coordinates} дуга кривой $f_k$, лежащая в области $E$ является графиком явной функции $s = f_{24}(q)$
\[f_{24}(q) = - \cfrac{1}{\tau^2} (2-3\tau^{2/3}q^{2/3}) 
+ \cfrac{2}{\tau^2} \left(1 - \tau^{2/3}q^{2/3}\right)^{\frac{3}{2}}. \]
\end{lemma}

Далее будем обозначать листы поверхности $root$, соответствующие нулям $y_{11}$ и $y_{10}$ как $rootl$ и $rootr$ соответственно. Эти поверхности в пространстве $\mathbb{R}^3(a, b, h)$ склеиваются в точках $(a, b, h)$ при $h = h(a, b, z_{cusp}(a, b)), a = f_k(b)$.
Заметим, что проводить все попарные проверки поверхностей первой серии и $rootl$, $rootr$ на совпадение значений $h$ не требуется.

\begin{lemma}\label{L:possible_pairs}
При $\varkappa >0$ достаточно проверить следующие пары поверхностей из разных серий:
\begin{itemize}
\item область I: \,\,$(rootr, -r)$, $(rootr, +r)$,
\item область II: $(rootr, -r)$, $(rootr, +r)$, $(rootr, +l)$ и $(rootl, -r)$.
\end{itemize}
\end{lemma}

Действительно, как  было установлено в \cite{Kozlov14}, $h_z' < 0$ при $0 < z < \sqrt[3]{b^2 c_1^2} = z_{cusp}$ и $h_z' >0$ при остальных $z \ne 0$. Отсюда $h(z_1)> h(rt) > h(lt) > h(rootr)$ и $h(+l) > h(rootl)$. Т.к. $h(-l) < h_l < h(rootl)$ в области II и $h_l < h(rootr)$ в областях I, II, то положение оставшихся точек относительно $rootr$ и $rootl$ известно. Отсюда следует утверждение леммы \ref{L:possible_pairs}.

Найдем уравнения разделяющих кривых $f_{20}, ..., f_{23}$, на которых достигаются равенства значений гамильтониана в точках $\pm r$ или $\pm l$ и $rootl, rootr$.

\begin{assertion}\label{A:curves_sep_set_1} Пусть $\varkappa >0$, тогда

1) равенства $h(rootr) = h(+r)$, $h(rootl) = h(-r)$ и $h(rootr) = h(-r)$ в координатах $(q, s) = (v^{-1} \sqrt{u}, v^{-1})$ достигается на графиках следующих функций $s = f_{i}(q)$ в плоскости $Oqs$, где $i \in \{20, 21, 22\}$ соответственно:
 \begin{equation}\label{Eq:sq_rootr_pr}
 f_{i}: \, \quad s(q) = \cfrac{-(1 + 2 \sigma_i \tau q) + \sqrt{(1 + 2 \sigma_i \tau q) ^2 - 2\tau^2 \left( 2 q^2  - 8\cfrac{q^2}{w_{i}(q)} + 128 \tau^2 \cfrac{q^4}{w_{i}^4(q)} \right)\mathstrut}}{\tau^2},
 \end{equation}

2) равенство $h(rootr) = h(+l)$ достигается на кривой, которая в координатах $(\widetilde q, \widetilde s)  = (u^{-1} \sqrt{v}, u^{-1})$ $\widetilde s = f_{23}(\widetilde q)$ является графиком функции $\widetilde{s} = f_{20}(\widetilde{q})$.

Здесь $\sigma_{20} = -1, \sigma_{21} = \sigma_{22} = +1$. В выражении $w_i$ знаки ``+'' и ``$-$'' перед корнем определяются выбором правого или левого корня $rootr$ и $rootl$ соответственно, а знаки ``+'' и ``$-$'' под корнем равны $-\sigma_i$ и зависят от выбора $+r$ или $-r$:
\[w_{20} = \left( 1 +  \sqrt{1 - 8 \sigma_{20} \tau \sqrt{u}/v } \right), w_{21, 22} = \left( 1 \pm  \sqrt{1 - 8 \sigma_{21} \tau \sqrt{u}/v } \right).\]
\end{assertion}

\begin{remark} При $1-8 \tau q <0,$  
выражения $w_{21}, w_{22}$ не определены. Значение $q = 1/(8\tau)$ соответствует общей точке $(u_0, f_{13}(u_0))$ кривых $f_{21}, f_{22}, f_{13}$ и $f_k$.
\end{remark}

Благодаря явному заданию кривых можно утверждать, что на рисунки \ref{Fig:Plane_Sep_uv_big} и \ref{Fig:sep_set_small} попали все ветви этих кривых, содержащиеся в $\Theta$.

\subsubsection{Дуги разделяющих кривых, входящие в $\Theta$} На следующем шаге для каждой из функций $f_i$ укажем подмножество промежутка $(x_1, x_2)$, на котором точка графика $f_i$ содержится в $\Theta$. Отметим, что общая тройка $(a, b, h)$ двух поверхностей содержится в $\mathbb{A}^2$, т.е. пара $(a, b) \in \Theta$, в точности если $\Sigma^{a, b}$ содержит две особые точки с равными значениями $h$ или осуществляется одна из перестроек бифуркационных диаграмм, происходящих при пересечении кривых $f_l, f_t, f_k, f_r, f_m$ или $b =0$.

Проверим, на каких дугах кривых $f_i$ лежат тройки, которым соответствуют две особые точки $\Sigma^{a, b}$ с равными абсциссами. Функция $f_i $ была записана в таблицу \ref{Tab:empty_curves}, если ее график содержит конечное или нулевое число точек из $\Theta$, или в таблицу \ref{Tab:sep_curves_2.1}, если целая дуга графика $f_i$ содержится в $\Theta$. Следующее утверждение доказано в пункте \ref{s:Proof_proof_A:Sep_set_1_arcs}.

\begin{assertion}\label{A:Sep_set_1_arcs}
В компактном случае Ковалевской разделяющее множество $\Theta$ состоит из следующих дуг кривых $f_i$:
\begin{enumerate}
\item $f_0$, $f_1$, $f_2$, $f_3$, $f_4$, $f_5$, $f_7$, $f_8$, $f_9$, $f_{10}$, $f_{14}$, $f_{15}$, $f_{16}$ --- с указанными в таблице \ref{Tab:sep_curves_2.1} областями определения в координатах $(u, v)$, заданных формулой~\eqref{Eq:uv_ab},
\item $v = f_{18}(u)$ --- при $u \in (0, \tau^2 /4)$,
\item $v = f_{13}(u)$ --- при $u \in (0, u_0)$. Здесь $u_0 = \cfrac{(5+3\sqrt{3})^2}{16}\tau^2$,
\item  $v = f_{17}(u)$ --- при $u \in (0, \tau^2/4)$,
\item $a = f_k(b)$  --- при $b \ge 0$.
\item Кривые $f_{6}, f_{11}, f_{12}, a = f_{19}(b)$ не добавляют новых дуг и точек в $\Theta$.

\item Описанные в утверждении \ref{A:curves_sep_set_1} кривые $f_{20}, f_{22}, f_{23}$ входят в $\Theta$ целиком, кривая $f_{21}$ --- до пересечения с $f_t$.
\end{enumerate}
\end{assertion}

\begin{remark}
Точка $(u_0, f_{13}(u_0))$ лежит на кривых $f_{13}$, где $h_{cusp} = h_{-r}$, и на кривой $f_k$, общей для поверхностей  $rootr$, $rootl$ и  $cusp$.
\end{remark}

Таким образом, на рисунках \ref{Fig:Plane_Sep_uv_big} и  \ref{Fig:sep_set_small} указаны все дуги разделяющих кривых, входящих в $\Theta$. В частности, отсутствуют иные ветви кривых $f_{20}, \dots, f_{23}$.

\subsubsection{Точки пересечения разделяющих кривых} Остается удостовериться, что на эти рисунки попали все точки пересечения этих дуг. В утверждении \ref{A:sepset_22} докажем, что все точки пересечения дуг $\Theta$, лежащие в достаточно малой окрестности вершины $2$, уже указаны на рисунке \ref{Fig:sep_set_small} (т.е. при рассмотрении такой окрестности ``под микроскопом'' не обнаружатся новые вершины графа $\Theta$).

\begin{assertion}\label{A:sepset_22}
Все пересечения кривых из $\Theta$ вблизи точки  $(\tau^2, \tau^2)$ указаны на рисунке \ref{Fig:sep_set_small}. А именно, кривые $f_r$, $f_{16}$ не имеют там точек пересечения с другими кривыми, кривая $f_8$ пересекает кривые $f_k, f_{23}, f_{t}$, кривая $f_{20}$ пересекает кривую $f_{t}$.
\end{assertion}

Теперь опишем поведение кривых $\Theta$ на большом удалении от нуля в плоскости $Ouv$.  В утверждении \ref{A:no_faraway_intersections} докажем, что дуги $\Theta$ не пересекаются при достаточно больших $(u, v)$, т.е. все вершины $\Theta$ попали на рисунок \ref{Fig:Plane_Sep_uv_big}.

\begin{assertion}\label{A:no_faraway_intersections}
Кривые $\Theta$ пересекают окружность достаточно большого радиуса с центром в начале координат в следующем порядке. Окружность обходится против часовой стрелки.
\[f_1 = f_l , f_2, f_{16}, f_{20}, f_{8}, f_{15}, f_{22}, f_{7}, f_{14} = f_t, f_{23}, f_{24} = f_{k}, f_{4} = f_r, f_{7}, f_{9} = f_m, f_{17}, f_{0}.\]
\end{assertion}

\begin{corollary}
Разделяющее множество имеет структуру графа с $75$ вершинами и $199$ ребрами, симметричного относительно прямой $b =0$.
Все вершины и ребра $\Theta$ изображены на рисунках \ref{Fig:Plane_Sep_uv_big}-\ref{Fig:sep_set_small}.
\end{corollary}

\section{Доказательства утверждений о разделяющем множестве}\label{s:Proof_proof}

\subsection{Доказательство лемм}\label{s:Proof_proof_Lemms}

\subsubsection{Явные формулы для поверхностей первой серии}\label{s:Proof_proof h_coord_points}

\begin{lemma} \label{L: h_coord_points}
Пусть $\varkappa >0$, тогда в области $\hat D = \{(a, b)\,|\, a > 2\sqrt{\varkappa} b > 0\}$ абсциссы особых точек первой серии в координатах $(u, v)$, заданных формулой \eqref{Eq:uv_ab}, имеют следующий вид, где $\tau = \sqrt{2}\varkappa c_1$:
\begin{equation} \label{Eq:h_coord_2_uv_pmrl}
h_{\pm l} = \frac {u \pm 2\tau \sqrt{v}} {2\varkappa}, \qquad
h_{\pm r} = \frac {v \pm 2\tau \sqrt{u}} {2\varkappa}, \qquad
h_{lt} = \frac {v} {u}\varkappa c_1^2 + \frac {u}{\varkappa}, \qquad
h_{rt} = \frac {u} {v}\varkappa c_1^2 + \frac {v}{\varkappa},
\end{equation}
\begin{equation} \label{Eq:h_coord_2_uv_rest}
h_{cusp} = 3 (u v)^{1/3} \left(\frac {c_1^2}{4\varkappa}\right)^{1/3}, \qquad
h_l = \varkappa c_1^2 + \frac{u}{\varkappa}, \qquad
h_{int} = \varkappa c_1^2 + \frac{u + v}{2\varkappa}.
\end{equation}
\end{lemma}

\begin{proof}
Несложно посчитать, что
\[ z_{\pm l} = \pm \sqrt{\frac v 2}c_1, \quad z_{lt} = \frac {u} {2\varkappa}, \qquad z_{\pm r} = \pm \sqrt{\frac u 2}c_1, \quad z_{rt} = \frac {v} {2\varkappa}, \quad z_{cusp} = \left(\frac {uv} {4} \cfrac{c_1^2}{\varkappa}\right)^{1/3}.\]

Для поверхностей класса $h = h(a, b, z_i(a, b))$ данные выражения были получены подстановкой $z_i(u, v)$ в формулу
 \[h(z) = \cfrac{b(u, v)^2 c_1^2}{z^2} + 2z.\]
Лемма \ref{L: h_coord_points} доказана.
\end{proof}

\subsubsection{Доказательство леммы \ref{L:singular_orbits} о сингулярных орбитах} \label{s:Proof_proof_L_singular_orbits}
\begin{enumerate}
 \item При $\varkappa >0$ из $f_1 = f_l(f_2)$ следует, что $x_i = \sqrt{\varkappa} J_i$, т.е. объединение сингулярных орбит изоморфно  $\mathbb{R}^3(\textbf{J})$ со структурой алгебры Ли $so(3)$. Функцией Казимира является $f = J_1^2 + J_2^2 + J_3^2 = b \varkappa^{-1/2}$. Интегралы $H, K$ и $f$ зависимы:
\[K = ((J_1 - \sqrt{\varkappa} c_1)^2 - J_2^2)^2 + 4 J_2^2 (J_1 - \sqrt{\varkappa} c_1)^2 = ((J_1 - \sqrt{\varkappa} c_1)^2 + J_2^2)^2.\]
\[\mathrm{При этом}\, H = 2(J_1^2 + J_2^2 + J_3^2) - J_1^2 - J_2^2 + 2c_1 x_1 = 2b \varkappa^{-1/2} - ((J_1 - \sqrt{\varkappa}c_1)^2 +\varkappa c_1^2 + J_2^2).\]
Тем самым,  $H = -\sqrt{K} + \varkappa c_1^2 + 2b \varkappa^{-1/2}$. Образ отображения момента $(H, K)$ одномерен и лежит на параболе $k = \left(h - a \varkappa^{-1} - \varkappa c_1^2\right)^2$.

 \item Интеграл $H(J_1, J_2)$ расслаивает $\mathbb{R}^3(\textbf{J})$ на соосные круговые цилиндры и особую прямую $J_1 = \sqrt{\varkappa}c_1, J_2 = 0$. Значение $b_0 = \sqrt{\varkappa^3 c_1^4}$ соответствует касанию оси цилиндров и орбиты $J_1^2 +J_2^2 + J_3^2 = b \varkappa^{-1/2}$.

Рассмотрим множество троек $(f_l(b), b, h)$, для которых $Q_{f_l(b), b, h}$ не диффеоморфно $S^1$ или $2 S^1$. Для каждой сингулярной орбиты $M^2_{a, b}$ таковыми являются две точки $p_1^{\pm}$ касания цилиндров и сферы, а при $b_0 < b$ также две точки $p_2^{\pm}$ пересечения оси цилиндров и сферы. Нетрудно найти их координаты:
\[p_1^{\pm}: \quad \left(\pm \varkappa^{-1/4} \sqrt{b}, 0, 0\right), \qquad p_2^{\pm}: \quad \left(\sqrt{\varkappa} c_1, 0, \pm \sqrt{b\varkappa^{-1/2} - \varkappa c_1^2}\right).\]

При $0 <b \le b_0$ минимум и максимум $H$ достигаются в точках $p_1^{-}$ и $p_1^{+}$ соответственно. При $b_0 < b$ слой, содержащий точку $p_1^{+}$, диффеоморфен окружности, а максимум $H$ достигается в точках $p_2^{\pm}$. При этом $K(p_1^{\pm}) >0$ и $K(p_2^{\pm}) = 0$.

 \item Из леммы \ref{L: h_coord_points} при $a = f_{l}(b)$ имеем равенства $z_{-r} = z_{-l}$, $z_{+l} = z_{+r}$, $z_{lt} = z_{rt}$. Явно проверено, что $H(p_1^{-}) = h_{-l}, H(p_1^{+}) = h_{+l}$ и $H(p_2^{\pm}) = h_l = h_{rt}$. Равенство $H(p_2^{\pm}) = h(z_1)$ по непрерывности следует из расположения $z_1$ между вершинами парабол \eqref{Eq:Left Parabola} и \eqref{Eq:Right Parabola}. Этих данных достаточно, чтобы описать $\mathbb{A}^2$ в окрестности плоскостей $a = f_l(b)$.
\end{enumerate}
Лемма \ref{L:singular_orbits} доказана.

\subsection{Доказательство утверждения \ref{A:Curves_sep_set_2}}\label{s:Proof_proof_A:Curves_sep_set_2} Достаточно сравнить значения функций $h_i(a, b)$ и $h(a, b, z_i(a, b))$ из леммы \ref{L: h_coord_points} для 36 пар поверхностей. 

\begin{enumerate}
 \item Сравнение значений $h$ поверхностей $\pm l, \pm r, lt, rt$ не представляет сложностей. Для сравнения $h_{cusp}$ и значения $h$ для поверхностей $\pm l, \pm r$ сделаем замену $w = w(u, v) >0$ и получим кубический многочлен с кратным корнем. Например, для $-l$\, замена и многочлен имеют вид $w = u^{1/3} v^{-1/6} \tau^{-1/3}$ и $(w-2)(w+1)^2$ соответственно. Для пар $(lt, cusp)$ и $(rt, cusp)$ используется похожая замена $w = u^{-2/3} v^{1/3} \tau^{2/3}$.

Значения функций $h_{+l}$ и $h_{cusp}$, $h_{+r}$ и $h_{cusp}$ равны на кривых $v = u^2/\tau^2, u > \tau^2$ и $u = v^2/\tau^2, 0 < v < \tau$ соответственно. Две эти кривые являются образом кривой $f_r$ в координатах $(u, v)$.

  \item Функции $h_l$ и $h_{+l}$ равны на кривой $a = f_t (b(t), t), b = b(t)$. Приняв за $t = \sqrt{u / (2\varkappa)}$, и подставив $a = a(u, v), b = b(u, v)$, получим уравнение кривой $f_{14}$.
Отметим, что  при $u \in (0, \tau^2)$ совпадают точки трех поверхностей: $+l$, $l$ и $z_1$. При $v < f_{14}(u), 0\le u < \tau^2$ точкам поверхностей $l$ и $z_1$ не соответствует особых точек $\Sigma^{u, v}$.

 Сравнение значений $h_l$ и $h_{int}$ друг с другом или с одной из функций $h_{-l}$, $h_{\pm r}$, $h_{lt}$, $h_{rt}$ не представляет сложностей. Отметим, что образ кривой $f_m$ в координатах $(u, v)$ является объединением отрезка $v = \tau^2, 0 \le u \le \tau^2$ и луча $u = \tau^2, v \ge \tau^2$, на которых равны значения пар функций $h_{int}$ и $h_{rt}$, $h_{int}$ и $h_{lt}$  соответственно.

 \item Кривая $v = f_{17}(u)$, в точках которой равны $h_l$ и $h_{cusp}$, и луч $v =u, u > 0$ не имеют общих точек кроме точки $(\tau^2, \tau^2)$, поскольку значение функции $u f_{17}(u) - u^2 > 0$ положительно в точках ее экстремумов $u = \tau^2$ и $u = \tau^2 /4$ и в точке $u =0$.

 \item Пару $cusp$ и $int$ рассмотрим в координатах $(a, b)$. Полученная кривая $a = f_{19}(b)$ не войдет в $\Theta$, т.к. она лежит ниже прямой $a = f_l(b)$ при $b >0, b \ne \varkappa^{3/2} c_1^2$:
\[f_{19}(b) = 3 b^{2/3} \varkappa c_1^{2/3} - \varkappa^{2} c_1^2 < 2\sqrt{\varkappa} b = f_l(b), \quad \textrm{т.к. } \quad 0 < (b^{1/3} +1/2 \varkappa^{1/2}c_1^{2/3})(b^{1/3} - \varkappa^{1/2}c_1^{2/3})^2.\]
Т.к. при $b >0$ полученный многочлен от $b^{1/3}$ неотрицателен, и равен нулю только при $b = \varkappa^{3/2} c_1^2$, то единственной точкой кривой $f_{19}$ в $\Theta$ является вершина $2$.
\end{enumerate}
Утверждение \ref{A:Curves_sep_set_2} доказано.

\subsection{Уравнения кривых $f_l, f_t, f_k, f_r, f_m$ в координатах $(q, s)$}\label{s:Proof_proof_L:bif_curves_sq}

\begin{enumerate}
\item  Подставив выражения $u(s, q)$ и $v(s)$ в уравнения кривых $f_1, f_4$,\, $f_9, f_{14}$ из таблицы \ref{Tab:sep_curves_2.1}, легко получить уравнения из таблицы \ref{Tab:sep_curves:q_s-type} для кривых $f_l, f_t, f_r$ и $f_m$.

\item  Проделав то же самое с уравнением кривой $f_k$ в координатах $(u, v)$, получим квадратное уравнение относительно $s$.
 \begin{equation}\label{Eq:Function_F_k_uv}
 \cfrac{u+v}{2} = \cfrac{3 (uv)^{2/3}}{4^{5/3} \varkappa^{2/3} c_1^{2/3}} + \cfrac{3 \varkappa c_1^{2/3}(uv)^{1/3} }{2(4\varkappa)^{1/3}} - \cfrac{1}{4}\varkappa^2 c_1^2
 \end{equation}

\[ \cfrac{\tau^2}{2}\, s^2 + s \left( 2 - 3 \tau^{2/3} q^{2/3} \right) + 2q^2 - \cfrac{3 q^{4/3}}{\, 2 \tau^{2/3}} = 0.\]
\[f_k: \quad s = f_k^{\pm}(q) = \cfrac{1 - 3x(q) \pm 2 x^{\frac{3}{2}}(q)}{\tau^2}, \quad \textrm{где}\,\, x(q) = 1 - q^{2/3}\tau^{2/3}.\]

\item Выбор знака перед радикалом, т.е. нужной ветви $f_{k}^{-}$ или $f_{k}^{+}$, определяется попаданием данной ветви в $E$, т.е. условием $f_m(q) \le f_{k}(q) \le f_l(q)$ на промежутке $q \in (0, 1]$.

Найдя корни уравнений $f_k^{\pm}(q) = f_l(q)$ и $f_k^{\pm}(q) = f_m(q)$ с учетом их кратностей, легко увидеть, что $f_k^{-}(q) \ge f_m(q) \ge f_{k}^{+}(q) \ge f_l(q)$ при $0 < q \le \tau^{-1}$. Значит, $f_k = f_k^{+}$.
\end{enumerate}
Лемма \ref{L:bif_curves_sq} доказана.

\subsection{Доказательство утверждения \ref{A:curves_sep_set_1}}\label{s:Proof_proof_A:curves_sep_set_1}

Подробно обоснуем формулы для кривой $f_{20}$, разделяющей точки $+r$ и $rootr$. Небольшие отличия для остальных случаев будут указаны в пункте 3.

 1. Приравняем $h(z) = h(z_{\delta r})$, тогда $z_{\delta r}$ и один из нулей параметрической кривой $z_{rootr}, z_{rootl}$ сопряжены как корни этого уравнения
\[\cfrac{u v c_1^2}{4 \varkappa z^2} + 2z = \cfrac{v + 2 \delta \tau \sqrt{u}}{2\varkappa}.\]
Разделим многочлен на двучлен $z - z_{\delta r} = \left(z - \delta c_1 \sqrt{u /2} \right)$ и вычислим корни квадратного трехчлена:
\[8\varkappa  z^3 - z^2 (2 v + 4\sqrt{2}\delta \varkappa c_1 \sqrt{u}) +uv c_1^2 = \left(z - \delta c_1 \sqrt{u /2}\right) (8 \varkappa z^2 - 2 v z - \sqrt{2}\delta c_1 \sqrt{u} v),\]
\[z_{root} = \cfrac{v}{8\varkappa} \left( 1 \pm \sqrt{1 + 8 \delta \tau \sqrt{u}/v }\right) = \cfrac{v}{8\varkappa} w_{x,+r}(u, v).\]

В выражении $w_{x, y}(u, v)$ индексы $x \in \{rootl, rootr\}$ и $y \in \{+l, -r, +r\}$ указывают, абсциссы какой пары особых точек $\Sigma^{a, b}$ были приравнены. Правило выбора знака в выражении $w(u, v)$ с индексами $x, y$ опишем в пункте~3. Получили, что  $w(u,v)$ является функцией одной переменной $q = \sqrt{u}/ v$ в случае $y = \pm r$ или одной переменной $\sqrt{v}/u$ в случае $y = +l$.

2.  Подставим найденные $z_{root}$ в уравнение $k(z) = 0$\, и разделим на $v^2$:
\[k(z) = \left(2 (u + v) c_1^2 - 4\cfrac{uv c_1^2}{4\varkappa z} + \cfrac{u^2 v^2 c_1^4}{(4\varkappa)^2 z^4}\right) -2 \varkappa c_1^2 h(z) + \varkappa^2 c_1^4 = 0,\]
\[2 c_1^2 \cfrac{u}{v^2}  + 2 c_1^2 \cfrac{1}{v}  - \cfrac{8 u  c_1^2}{v^2 w(q)} + \cfrac{u^2 c_1^4 8^4\varkappa^4}{(4\varkappa)^2 v^4 w^4(q)} -2 \varkappa c_1^2 \cfrac{1}{2\varkappa v} -2 \sqrt{2} \delta \varkappa c_1^3\cfrac{\sqrt{u}}{v^2}  + \cfrac{\varkappa^2 c_1^4}{v^2} =0.\]

Перепишем уравнение в координатах $(q, s)$:
\[\varkappa^2 c_1^2 s^2  + (1 -2 \delta \tau q) s + 2 q^2  - 8\cfrac{q^2}{w(q)} + 128 \tau^2 \cfrac{q^4}{w^4(q)}     =0\]

\[s = \cfrac{-(1 -2 \delta \tau q) \pm \sqrt{(1 -2 \delta \tau q) ^2 - 2 \tau^2 \left( 2 q^2  - 8\cfrac{q^2}{w_{x, +r}(q)} + 128 \tau^2 \cfrac{q^4}{w_{x, +r}^4(q)} \right)\mathstrut}}{\tau^2 }\]
Здесь $w = w_{x, +r} = \left( 1 \pm  \sqrt{1 + 8 \delta \tau q } \right)$. Получена явная формула кривой $s(q)$, содержащей те и только те точки $(q,s)$ для которых абсцисса корня совпадает с абсциссой $z_{+r}$.

 3. Опишем выбор знаков для $s(q)$, $w_{x, y}(q)$ и $\sigma_i$. Достаточно описать пары $(x, y) \in \{(rootr, +r), (rootl, -r), (rootr- r)\}$, т.к. для пары $(rootr, +r)$ вычисления проходятся аналогично случаю $(rootr, +r)$, в координатах $(\widetilde q, \widetilde s) = (q(v, u), s(u))$.

    Из пункта 1 видно, что знак $\delta_i = \textrm{sgn}\,z_{\pm r}$, т.е. $\sigma_i = -\delta_i, w_i(q) = 1 \pm \sqrt{1 - 8 \sigma_i \tau q }$. Т.е. под радикалом функции $w_{x, y}(q)$ выбирают знак "плюс"\, для точки $+r$ и "минус"\, для $-r$.

    Знак перед радикалом в $w_{x, -r}(q)$ определяет выбор корня --- знак "плюс"\, для $rootr$ и знак "минус"\, для $rootl$. Отметим, что $h(+r) > h(rootl)$, поэтому этой паре не соответствует дуг $\Theta$.

    Знак перед внешним радикалом в итоговой функции $s(q)$ определяется двумя требованиями: неотрицательностью подкоренного выражения и принадлежностью кривой области $E$. Таким образом, перед внешним корнем надо выбрать знак "плюс"\, для всех трех кривых $f_{20}, f_{21}, f_{22}$.

    Получили указанные в утверждении выражения для $\sigma_i, w_i(q)$ и $f_i(q), i \in \{20, 21, 22\}.$ Утверждение \ref{A:curves_sep_set_1} доказано.

\subsection{Доказательство утверждения \ref{A:Sep_set_1_arcs}}\label{s:Proof_proof_A:Sep_set_1_arcs}
Кривые $f_i, i \in \{1, 4, 5, 9, 10, 14\}$ с указанными в таблице \ref{Tab:sep_curves_2.1} областями определения являются образами кривых $f_l, f_t,$ $f_r, f_m$ при замене $(a, b) \rightarrow (u, v)$, и потому входят в $\Theta$.

Остальные кривые $f_i$ соответствуют парам пересекающихся поверхностей. Проекция кривой их пересечения или касания на $Oab$ имеет конечное число общих точек с границами областей I-IX. Для каждой кривой назовем подходящими те области, где обе особые точки входят в $\Sigma^{a, b}$. Дуга кривой, лежащая в такой области, войдет в разделяющее множество $\Theta$.

\begin{enumerate}
 \item Для кривых $f_2, f_3, f_7, f_8, f_{15}$ подходящими являются области I-V, V, I-IV, I-III, I-V. На всей области определения, указанной в таблице \ref{Tab:sep_curves_2.1}, каждая из кривых содержится в замыкании своих подходящих областей. Достаточно проанализировать порядок роста данных кривых, а для кривой $f_8$ --- значения $f_t', f_k', f_8'$ и $f_r'$ в точке $(\tau^2, \tau^2) \in Ouv$.

Для кривой $f_{16}$ области I, VI и VII являются подходящими. Кривая лежит именно в них, поскольку $f_{t}(u) = f_{14}(u) < f_{16}(u) < \tau^2$ при $0 < u < \tau^2$ и $f_l(u) < f_{16}(u) < f_t(u)$ при $u > \tau^2$. Это легко обосновать, рассмотрев $f_{14}' - f_{16}'$ при $u < \tau^2$ и $f_{16}'''(u), f_{14}'''(u)$ при $u = \tau^2 +0$.

 \item Для кривой $f_{18}$ подходящей является область VII. Данная кривая $v = \tau^2 /2$ пересекает кривую $f_r: v(u) = \tau\sqrt{u}$ в точке $(\tau^2/4, \tau^2/2)$, т.е. при $u < \tau^2/4$ график $f_{18}(u)$ лежит в VII.

 \item На кривой $f_{13}$ совпадают абсциссы особых точек $-r$ и $cusp$. Подходящими областями являются III, IV, V, VII, IX. При $0 < u < \tau^2$ имеем $f_{13}(u) > f_5(u)$, т.е. график лежит в подходящих областях V, VII и IX. Порядок роста функции $f_{13}$ равен $1/2$, т.е. график пересечет кривые $f_4$,  $f_k$ и $f_{14}$. Найдем точку пересечения $u_0$ кривых $f_{13}$ и $f_k$, подставив $v = 8 \tau \sqrt{u}$ в \eqref{Eq:Function_F_k_uv}:
\begin{center}
$u+8\tau \sqrt{u} = 3 u + 3 \tau \sqrt{u}  - \cfrac{\tau^2}{4}$

$ 2u - 5 \tau \sqrt{u} - \cfrac{\tau^2}{4} = 0, \qquad   u_0
=  \cfrac{(5 + 3\sqrt{3})^2 \tau^2}{16}.$
\end{center}

 \item В точках пересечения кривых $f_{17}$ и $f_k$ имеем $h_{cusp} = h_l$ и $cusp = rootl = rootr$. Значит, в ней $h_{rootl} = h_l$, и эта точка принадлежит $f_t$. Точка $(\tau^2, \tau^2)$ является единственной общей точкой трех кривых $f_t, f_{17}, f_k$, т.е. график $f_{17}$ лежит ниже графика $f_{k}$ и не попадает в подходящие области (III и IV) при $u > \tau^2$.

Кривая $f_{17}$ пересекает кривую $f_r$ в точке $(\tau^2/2, \tau^2/ \sqrt{2})$, области V и VII являются подходящими, а область VIII нет.

\item  Дуга $v < \tau^2$ кривой $f_6$ находится в областях VI и VIII, которые не являются для нее подходящими: при $v < \tau^2$ имеем $f_r(u) > f_{6}(u) > u$.

\begin{equation}
f_{r}(v) = \cfrac{v^2}{\tau^2} \lor \cfrac{2v^{3/2}}{\sqrt{v}+\tau^2} = f_{6}(v) \quad \Rightarrow \quad  (v + \sqrt{v} \tau^2  - 2 \tau^4) v^{3/2} \lor 0.
\end{equation}
Аналогично, при $0 < u <\tau^2 /2$ график $v = f_{11}(u)$ лежит ниже прямой $v =0$, а при $\tau^2/2 < u <\tau^2$ график лежит выше прямой $v = \tau^2$, т.е. в не подходящей для $f_{11}$ области V. Из утверждения \ref{A:Curves_sep_set_2} следует, что график кривой $a = f_{19}(b)$ лежит вне $D$.

 \item Функции $f_{20}, f_{22}, f_{24}$ определены при сколь угодно близких к нулю значениях $q$, т.е. при сколь угодно больших $v$. Эти кривые $s = f_i(q)$ лежат выше кривой $s = f_l(q) = q^2$. Равенство абсциссы $h$ для нуля параметрической кривой и одной из точек типов $-r, +r, +l$ возможно только при $f_l(b) \le a \le f_k(b)$, т.к только в областях I и II параметрическая кривая \eqref{Eq:Parametric Curve} имеет нули, отличные от $z_1$. Напомним, точка $z_1$ на плоскости $Ohk$ всегда лежит правее остальных особых точек.

  Для кривой $f_{21}$ данный факт следует из рассмотрения области определения радикала. Самая правая на плоскости $Osq$ точка кривой $s = f_{21}(q)$ лежит на кривой $f_t$.
\end{enumerate}
Утверждение \ref{A:Sep_set_1_arcs} доказано.

\subsection{Доказательство утверждения \ref{A:sepset_22}}\label{s:Proof_proof_A:sepset_22}

\begin{enumerate}
 \item В координатах $(u, v)$ кривые $f_{24} = f_k, f_{20}$ и $f_{23}$, выходящие из вершины $2$ c координатами $(\tau^2, \tau^2)$, заданы неявно. Вид разделяющего множества для ``средних'' по величине значений $(u, v)$, изображенных на рисунке \ref{Fig:Plane_Sep_uv_big}, достоверен. Устройство $\Theta$ в окрестности этой вершины $\Theta$ докажем аналитически.

На кривой $f_{20}$ равны $h_{rootr}$ и $h_{+l}$, на кривой $f_{23}$ --- $h_{rootr}$ и $h_{+r}$. При этом $h(+r) = h(+l)$ на кривой $f_3$ при $u \le \tau^2$. Т.е. $f_{20}$ и $f_{23}$ имеют одну общую точку $(\tau^2, \tau^2)$.

Кривая $f_{23}$ может пересечь $f_{k}$ только в точке $(u, v)$, для которой $h(+l) = h(cusp)$, т.е. на кривой $f_4 = f_r$ при $u \ge \tau^2$. Значит, $f_k$ лежит выше кривой $f_{23}$ при $u > \tau^2$.  Т.к. $f_{20}$ пересекает $f_{23}$ только в вершине $2$, то она тоже не пересекает $f_k$ при $u >\tau^2$.

\item  Через точку пересечения $f_8$ и $f_4$ пройдет кривая $f_3$. Т.е. это только точка $(\tau^2, \tau^2).$

Пару кривых $f_8$ и $f_k$ рассмотрим в координатах $(q, s)$. Уравнение $f_k$ выведено в лемме \ref{L:bif_curves_sq}, а уравнение $f_8$ несложно найти: $s >0$, и в следующем уравнении требуется выбрать знак $+$
\[s = f_8(q) = \cfrac{q}{2\tau}\left(-1 \pm \sqrt{1 + 8 \tau q}\right).\]

Численное решение уравнения $f_8(q) = f_k(q)$ дает корень $q_0 = 1/\tau$ и корень $q_1 \approxeq 0,57$, отделенный от нуля и $q_0$. При этом в точке $q_0$ имеем $f_8'(q_0-0) = 5\sqrt{2}/6 < \sqrt{2} = f_k'(q_0-0)$. Значение $q_0' \approxeq 0, 65$ является ближайшей к $q_0$  точкой равенства $f'_k$ и $f_8'$.

Из теории вычислительных методов следует, что достаточно исследовать достаточно малую окрестность $q_0$ на наличие общих точек этих кривых. Поскольку производные непрерывны, достаточно показать сохранение знака $f_k'(q)- f_8'(q)$ в некоторой окрестности.

Докажем, что на $(q_0', q_0)$ нет нулей второй производной, т.е. там $f_k''(q) > f_8''(q)$. Используем оценку:

Вторая производная $f_k''(q)$ имеет асимптотику $(q-q_0)^r, -1 < r < 0$, а $f_8''(q)$ определена и конечна в этой точке.

Для всех $q$ из промежутка $0, 69 = p <q < q_0$ верна оценка \[f_8''(q) < -\cfrac{8 p \sqrt{2}}{\left( 1 + 8 q_0 \sqrt{2} \right)} + \cfrac{4}{p \sqrt{1 + 8 q_0 \sqrt{2}}} < \cfrac{2^{2/3}}{3 q_0 \sqrt{1 - 2^{1/3} p^{2/3}}} - \cfrac{2}{p^{4/3}} < f_k''(q). \]

С помощью сеточного разбиения и оценки на третьи производные доказывается отсутствие нуля $f_k''-f_8''$ на промежутке $[q_0', p]$. Поскольку между нулями непрерывной $f'$ лежит ноль $f''$, то на промежутке $q_1, q_0$ разность $f_k - f_8$ не обращается в ноль и не меняет знака.
\end{enumerate}
Утверждение \ref{s:Proof_proof_A:sepset_22} доказано.

\subsection{Доказательство утверждения \ref{A:no_faraway_intersections}}\label{s:Proof_proof_A:no_faraway_intersections}

\begin{enumerate}
 \item Кривые $f_l, f_t, f_k, f_r, f_m$ не пересекаются друг с другом при больших значениях $u^2 + v^2$. Кривые $f_7$ при $u \rightarrow \tau^2 + 0$ и $f_{17}$ при $u \rightarrow +0$ неограниченно возрастают, приближаясь справа к прямым $f_m$ и $u =0$ соответственно.

 Все кривые, возможно, кроме $f_k$ и $f_{20}, ..., f_{23}$, монотонно возрастают при больших $u$ с известной скоростью. Кривая $f_7$ не пересекается с $f_{15}$ при больших $u$, т.к. $f_7(u) -  f_{15}(u) = const_7 - const_{15} + o(1)$, где $const_7 = 2\tau^2 > \tau^2 = const_{15}$. Аналогично в случае кривых $f_8$ и $f_{16}$.

Кривая $f_2$ имеет порядок роста как $u + const \sqrt{u} + ...$. Начиная с некоторого $u$ кривая $f_2$ лежит ниже всех остальных кривых, значит, это свойство сохранится далее из монотонности кроме $f_l$ на большом удалении $u^2 +v^2$.

\item  Докажем, что кривая $f_{22}$ лежит между $f_7$ и $f_{15}$ при больших u. Все три кривые неограниченно возрастают при $q \rightarrow 0$.
На кривой $f_7$ равны абсциссы $h_l = h_{-r}$, на кривой $f_{15}$ --- $h(z_{lt} = h_{-r}$, на кривой $f_{22}$ имеем $h_{rootr} = h_{-r}$. При этом точка $rootr$ лежит левее точки $z_{lt}$ и правее точки $h_l$. Значит, данные кривые не могут пересекаться. Аналогично кривая $f_{20}$ лежит между $f_8$ и $f_{16}$ при больших $u$.

\item Кривая $f_{23}$, в точках $(u, v)$ которой $h_{rootr} = h_{+l}$, лежит на $Ouv$ между $f_k$ и $f_t$ при всех отмеченных на чертеже $u$. На кривой $f_k$ имеем совпадение особых точек:  $rootr = rootl = cusp$. При этом кривая $h(z_{+l}) = h(z_{cusp})$ есть кривая $f_r$ при $u > \tau^2$. Аналогично, на кривой $f_t$ $+l = h_l = rootl$. Поскольку $rootr = rootl$ есть кривая $f_k$, то кривые $f_t$ и $f_{20}$ не пересекаются при $u > \tau^2$.

 \item Кривые $f_t, f_r$ имеют квадратичный порядок роста $u^2$, $f_k$ находится между ними, потому все три кривые повторно не пересекутся достаточно далеко с кривыми линейного порядка роста.
\end{enumerate}

Утверждение \ref{A:no_faraway_intersections} доказано.

\clearpage

\begin{table}
\centering
\begin{tabular}{c}
\includegraphics[width=\linewidth]{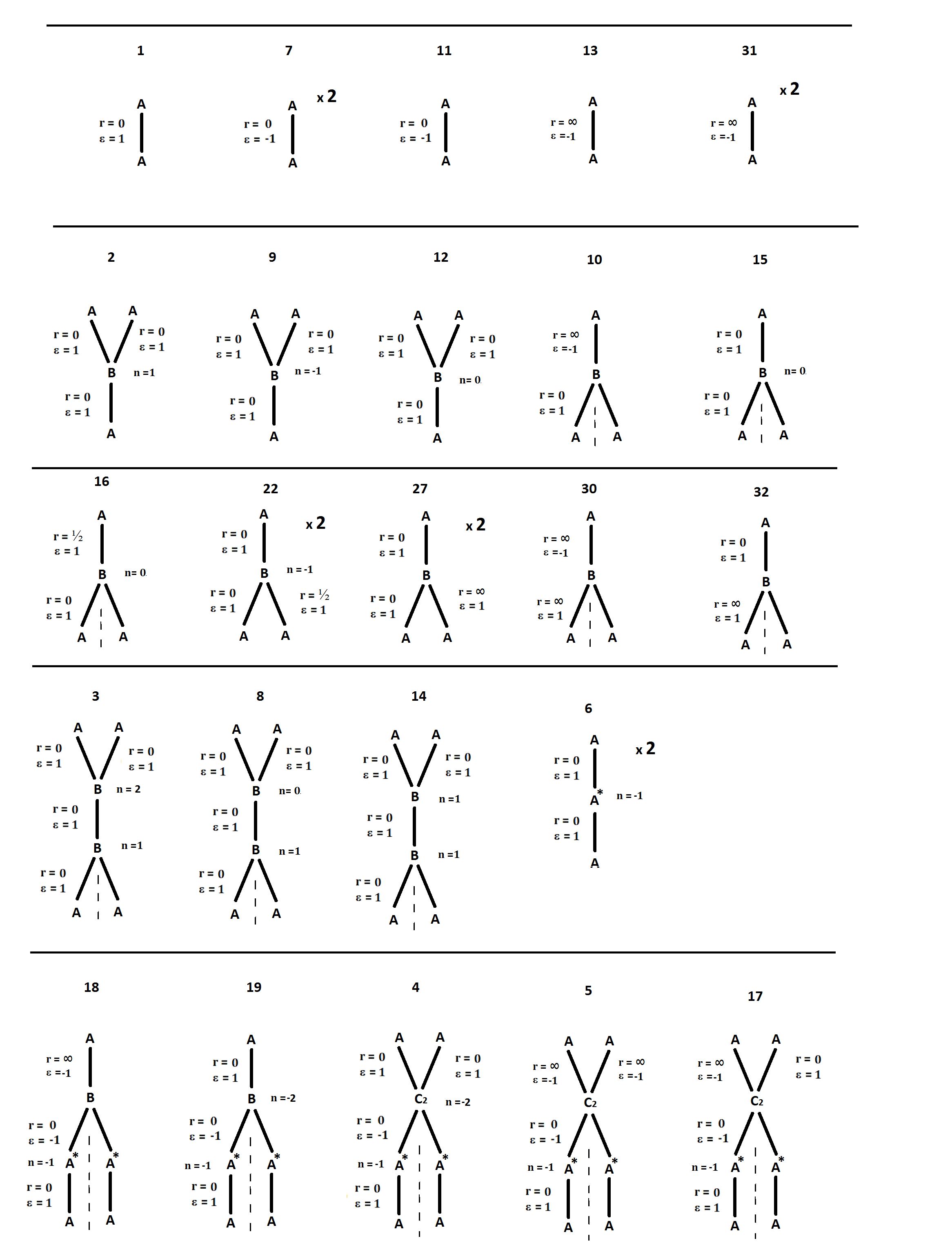}
\end{tabular}
   \caption{Список молекул с матрицами склейки. Часть 1.} \label{Tab:Molecules_classes_labeled_1}
\end{table}

\begin{table}
\centering
\begin{tabular}{c}
\includegraphics[width=1.1\linewidth]{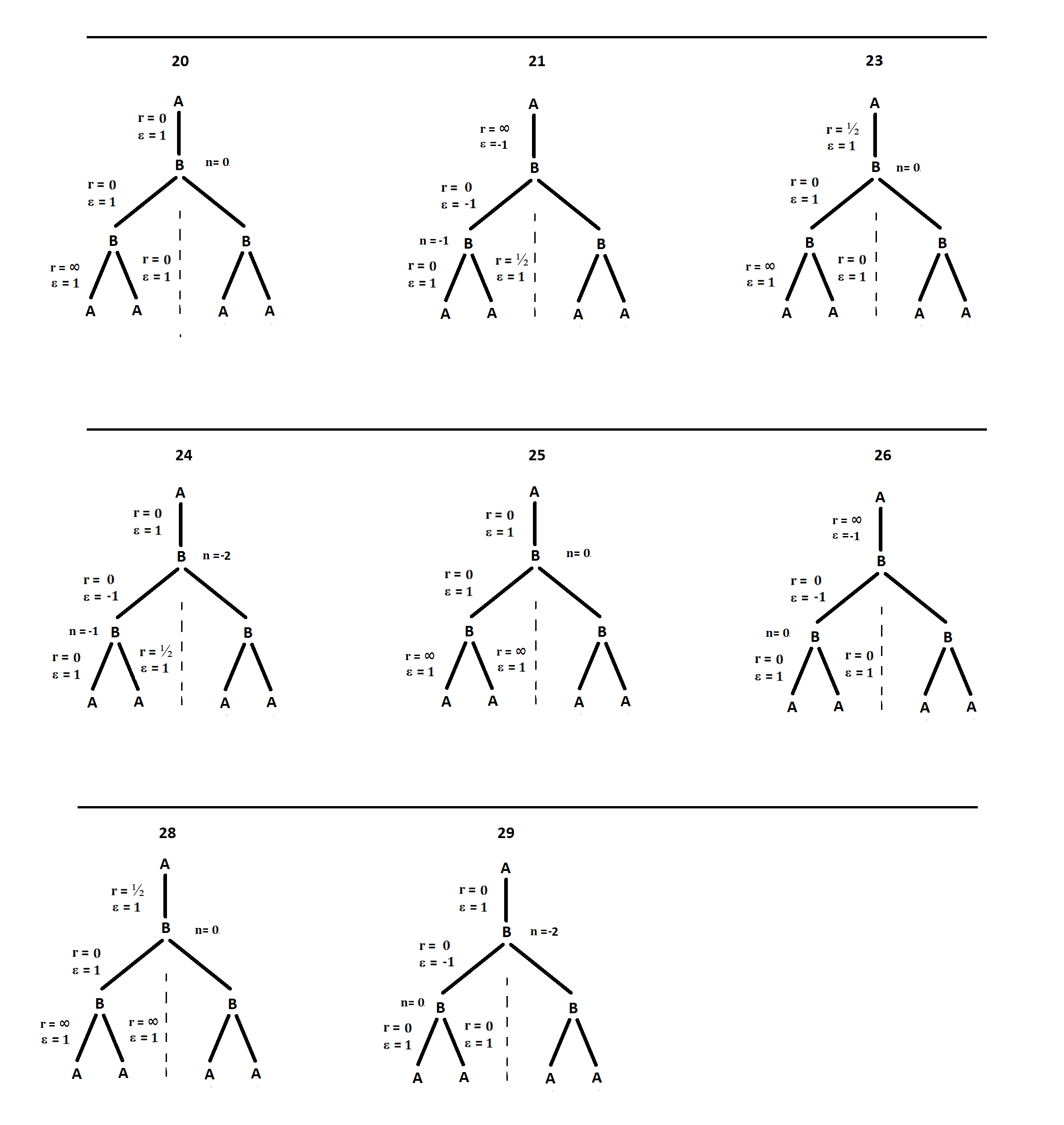}
\end{tabular}
   \caption{Список молекул с матрицами склейки. Часть 2.} \label{Tab:Molecules_classes_labeled_2}
\end{table}

\begin{table}
\centering
\begin{tabular}[t]{|l|c|c|c|c|c|c|c|c|c|c|c|c|c|c|c|c|c|c|c|c|c|c|c|c|c|c|c|}
\hline
класс $L_i$ & 1      & 2   & 3 & 4 & 5 & 6 & 7 & 8&  9&      10& 11& 12 & 13\\
\hline
молекула & 1, 7, 11 & 2, 9 & 3 & 4 & 5 & 6 & 8 & 10& 12, 15& 13, 31& 14 & 16 & 17\\
\hline
\end{tabular}
\begin{tabular}[t]{|l|c|c|c|c|c|c|c|c|c|c|c|c|c|c|c|c|c|c|c|c|c|c|c|c|c|c|c|}
\hline
класс $L_i$& 13 & 14 & 15& 16& 17& 18& 19 & 20 & 21& 22 & 23 & 24 & 25 & 26 & 27  \\
\hline
молекула & 17 & 18& 19& 20& 21& 22& 23 & 24 & 25& 26 & 27 & 28 & 29 & 30 & 32 \\
\hline
\end{tabular}
\caption{Классификация слоений на связных компонентах $Q^3_{a, b, h}$.}
\label{Tab:Liouville_classes_graphs}
\end{table}

\begin{table}
\centering
\begin{tabular}[t]{|c|c|c|c|}
\hline
 символ & дуга бифуркационной диаграммы & атом & семейство торов \\
\hline
$\xi_1$ & $(z_4, z_3), (z_4, z_5), (z_4, z_{11}), (z_4, z_8), (z_7, z_5), (z_7, z_8)$   & A & (1)  \\
\hline
$\xi_2$ & $(z_3, z_2), (z_5, z_2)$ & 2 A & (3) \\
\hline
$\xi_3$ & $(z_2, z_1), (z_{10}, z_1)$ & 2 A & (2) \\
\hline
$\xi_4$ & $(z_6, z_5), (z_6, z_8), (z_7, z_5), (z_7, z_8)$   & A & (4) \\
\hline
$\xi_5$ & $(z_8, z_9), (z_8, z_{10}), (z_{11}, z_{10}), (z_{11}, z_{9})$ & A & (1) \\
\hline
\end{tabular}\caption{Новые семейства дуг бифуркационных диаграмм.}
\label{Table_arcs}
\end{table}

\begin{table}
\begin{tabular}[t]{|c||c|c|c|c|c|c|c|c|c|c|c|c|c|c|c|c|c|c|}
\hline
X.1  & $y_{1}$ & 1 & $y_4$ & 2 &$y_{2}$& 3 & $y_{3}$ & 4 & $z_7$ & 5 & $z_{5}$ & 6 & $z_{2}$ & 7 & $z_1$ \\
\hline
X.2  & $y_{1}$ & 1 & $y_4$ & 2 &$y_{2}$& 3 & $z_7$ & 8 & $y_{3}$ & 5 & $z_{5}$ & 6 & $z_{2}$ & 7 & $z_1$ \\
\hline
\hline
XI.1 & $y_{1}$ & 1 & $y_4$ & 2 & $z_7$ & 9 &$y_{2}$& 8 & $z_{8}$ & 10 & $z_{10}$ & 7 & $z_1$ & &  \\
\hline
XI.2 & $y_{1}$ & 1 & $y_4$ & 2 & $z_7$ & 9 & $z_{8}$ & 11 &$y_{2}$& 10 & $z_{10}$ & 7 & $z_1$ & & \\
\hline
\hline
XII  & $y_{1}$ & 1 & $y_4$ & 2 & $z_7$ & 9 & $z_{8}$ & 11 & $z_{9}$ &  & & & & &  \\
\hline
\end{tabular}
\caption{Промежутки X-XII: порядок 3-камер и особых точек.}
\label{Tab:order_region_b0}
\end{table}

\begin{table}
\begin{tabular}[t]{|c||c|c|c|c|c|c|c|c|c|c|c|c|c|c|c|c|c|c|c|c|c|}
\hline
IX.1 & $y_{1}$ & $1$ & $y_{6}$ & $2$ & $z_{4}$ & $12$ & $z_{6}$ & $9$ & $z_{8}$ & $11$ & $z_{9}$ & & & &  \\
\hline
IX.2 & $y_{1}$ & $1$ & $z_{4}$ & $13$ & $y_{6}$ & $12$ & $z_{6}$ & $9$ & $z_{8}$ & $11$ & $z_{9}$ & & & &  \\
\hline
\hline
VIII & $y_{1}$ & $1$ & $z_{4}$ & $13$ & $z_{11}$ & $11$ & $z_{9}$ & & & & & & & &  \\
\hline
\hline
VII.1 & $y_{1}$ & $1$ & $y_{6}$ & $2$ & $z_{4}$ & $12$ & $z_{6}$ & $9$ & $z_{8}$ & $11$ &$y_{2}$& $10$ & $z_{10}$ & $7$ & $z_1$   \\
\hline
VII.2 & $y_{1}$ & $1$ & $z_{4}$ & $13$ & $y_{6}$ & $12$ & $z_{6}$ & $9$ & $z_{8}$ & $11$ &$y_{2}$& $10$ & $z_{10}$ & $7$ & $z_1$   \\
\hline
VII.3 & $y_{1}$ & $1$ & $y_{6}$ & $2$ & $z_{4}$ & $12$ & $z_{6}$ & $9$ &$y_{2}$& $8$ & $z_{8}$ & $10$ & $z_{10}$ & $7$ & $z_1$   \\
\hline
VII.4 & $y_{1}$ & $1$ & $z_{4}$ & $13$ & $y_{6}$ & $12$ & $z_{6}$ & $9$ &$y_{2}$& $8$ & $z_{8}$ & $10$ & $z_{10}$ & $7$ & $z_1$   \\
\hline
VII.5 & $y_{1}$ & $1$ & $y_{6}$ & $2$ & $z_{4}$ & $12$ &$y_{2}$& $14$ & $z_{6}$ & $8$ & $z_{8}$ & $10$ & $z_{10}$ & $7$ & $z_1$   \\
\hline
VII.6 & $y_{1}$ & $1$ & $z_{4}$ & $13$ & $y_{6}$ & $12$ &$y_{2}$& $14$ & $z_{6}$ & $8$ & $z_{8}$ & $10$ & $z_{10}$ & $7$ & $z_1$   \\
\hline
VII.7 & $y_{1}$ & $1$ & $z_{4}$ & $13$ &$y_{2}$& $15$ & $y_{6}$ & $14$ & $z_{6}$ & $8$ & $z_{8}$ & $10$ & $z_{10}$ & $7$ & $z_1$   \\
\hline
\hline
VI.1 & $y_{1}$ & $1$ & $z_{4}$ & $13$ & $z_{11}$ & $11$ &$y_{2}$& $10$ & $z_{10}$ & $7$ & $z_1$ & & & &   \\
\hline
VI.2 & $y_{1}$ & $1$ & $z_{4}$ & $13$ &$y_{2}$& $15$ & $z_{11}$ & $10$ & $z_{10}$ & $7$ & $z_1$ & & & &  \\
\hline
\end{tabular}
\caption{Области VI-IX: порядок 3-камер и особых точек.}
\label{Tab:order_region_small}
\end{table}

\begin{table}
\begin{tabular}[t]{|c||c|c|c|c|c|c|c|c|c|c|c|c|c|c|c|c|c|c|c|c|c|c|c|c|c|c|c|c|c|c|c|c|c|}
\hline
V.3 & $y_{1}$ &  $1$ & $z_{4}$ & $13$ &$y_{2}$& $15$ & $y_{6}$ & $14$ & $z_{6}$ & $8$ & $y_{3}$ & $5$ & $z_{5}$ & $6$ & $z_{2}$ & $7$ & $z_1$  \\
\hline
V.4 & $y_{1}$ &  $1$ & $z_{4}$ & $13$ &$y_{2}$& $15$ & $y_{6}$ & $14$ & $y_{3}$ & $17$ & $z_{6}$ & $5$ &$z_{5}$   & $6$ & $z_{2}$ & $7$ & $z_1$ \\
\hline
V.5 &$y_{1}$ &  $1$ & $y_{2}$& $16$ & $z_{4}$ & $15$ & $y_{6}$ & $14$ & $y_{3}$ & $17$ & $z_{6}$ & $5$ & $z_{5}$   & $6$ & $z_{2}$ & $7$ & $z_1$ \\
\hline
V.6 &$y_{1}$ &  $1$ & $y_{2}$& $16$ & $z_{4}$ & $15$ & $y_{6}$ & $14$ & $z_{6}$ & $8$ & $y_{3}$ & $5$ & $z_{5}$    & $6$ & $z_{2}$ & $7$ & $z_1$\\
\hline
V.7 & $y_{1}$ &  $1$ & $y_{2}$& $16$ & $y_{6}$ & $3$ & $z_{4}$ & $14$ & $z_{6}$ & $8$ & $y_{3}$ & $5$ & $z_{5}$    & $6$ & $z_{2}$ & $7$ & $z_1$\\
\hline
V.9 &$y_{1}$ &  $1$ & $y_{2}$& $16$ & $y_{6}$ & $3$ & $z_{4}$ & $14$ & $y_{3}$ & $17$ & $z_{6}$ & $5$ & $z_{5}$    & $6$ & $z_{2}$ & $7$ & $z_1$\\
\hline
V.12& $y_{1}$ &  $1$ & $y_{2}$& $16$ & $y_{6}$ & $3$ & $y_{3}$ & $4$ & $z_{4}$ & $17$ & $z_{6}$ & $5$ & $z_{5}$    & $6$ & $z_{2}$ & $7$ & $z_1$\\
\hline
V.2 & $y_{1}$ &  $1$ & $z_{4}$ & $13$ & $y_{6}$ & $12$ &$y_{2}$& $14$ & $z_{6}$ & $8$ & $y_{3}$ & $5$ & $z_{5}$    & $6$ & $z_{2}$ & $7$ & $z_1$\\
\hline
V.1 & $y_{1}$ &  $1$ & $y_{6}$ & $2$ & $z_{4}$ & $12$ &$y_{2}$& $14$ & $z_{6}$ & $8$ & $y_{3}$ & $5$ & $z_{5}$   & $6$ & $z_{2}$ & $7$ & $z_1$\\
\hline
V.8 & $y_{1}$ &  $1$ & $y_{6}$ & $2$ &$y_{2}$& $3$ & $z_{4}$ & $14$ & $z_{6}$ & $8$ & $y_{3}$ & $5$ & $z_{5}$   & $6$ & $z_{2}$ & $7$ & $z_1$\\
\hline
V.10& $y_{1}$ &  $1$ & $y_{6}$ & $2$ &$y_{2}$& $3$ & $z_{4}$ & $14$ & $y_{3}$ & $17$ & $z_{6}$ & $5$ & $z_{5}$   & $6$ & $z_{2}$ & $7$ & $z_1$ \\
\hline
V.11& $y_{1}$ &  $1$ & $y_{6}$ & $2$ &$y_{2}$& $3$ & $y_{3}$ & $4$ & $z_{4}$ & $17$ & $z_{6}$ & $5$ & $z_{5}$  & $6$ & $z_{2}$ & $7$ & $z_1$  \\
\hline
\end{tabular}
\caption{Область V: порядок 3-камер и особых точек.}
\label{Tab:order_region_V}
\end{table}

\begin{table}
\begin{tabular}[t]{|c||c|c|c|c|c|c|c|c|c|c|c|c|c|c|c|c|c|c|c|c|c|c|c|c|}
\hline
\hline
IV.1 & $y_{1}$ &  $1$ & $z_{4}$ & $13$ &$y_{2}$& $15$ & $y_{6}$ & $14$ & $y_{3}$ & $17$ & $y_{5}$ & $18$ & $z_{3}$  & $6$ & $z_{2}$ & $7$ & $z_1$   \\
\hline
IV.2 & $y_{1}$ &  $1$ &$y_{2}$& $16$ & $z_{4}$ & $15$ & $y_{6}$ & $14$ & $y_{3}$ & $17$ & $y_{5}$ & $18$ & $z_{3}$   & $6$ & $z_{2}$ & $7$ & $z_1$  \\
\hline
IV.3 & $y_{1}$ &  $1$ &$y_{2}$& $16$ & $y_{6}$ & $3$ & $z_{4}$ & $14$ & $y_{3}$ & $17$ & $y_{5}$ & $18$ & $z_{3}$    & $6$ & $z_{2}$ & $7$ & $z_1$ \\
\hline
IV.4 & $y_{1}$ &  $1$ & $y_{2}$& $16$ & $y_{6}$ & $3$ & $y_{3}$ & $4$ & $z_{4}$ & $17$ & $y_{5}$ & $18$ & $z_{3}$    & $6$ & $z_{2}$ & $7$ & $z_1$ \\
\hline
IV.5 & $y_{1}$ &  $1$ & $y_{2}$& $16$ & $y_{6}$ & $3$ & $y_{3}$ & $4$ & $y_{5}$ & $19$ & $z_{4}$ & $18$ & $z_{3}$   & $6$ & $z_{2}$ & $7$ & $z_1$ \\
\hline
\end{tabular}
\caption{Область IV: порядок 3-камер и особых точек.}
\label{Tab:order_region_IV}
\end{table}

\begin{table}
\begin{tabular}[t]{|c||c|c|c|c|c|c|c|c|c|c|c|c|c|c|c|c|c|c|c|c|c|c|c|c|}
\hline
III.1 & $y_{1}$ &  $1$ &  $z_{4}$ & $13$ &$y_{2}$& $15$ & $y_{9}$ & $20$ & $y_{7}$ & $21$ & $z_{3}$ &  $22$ & $y_{8}$    & $6$ & $z_{2}$ & $7$ & $z_1$ \\
\hline
III.2 & $y_{1}$ &  $1$ &  $z_{4}$ & $13$ &$y_{2}$& $15$ & $y_{9}$ & $20$ & $y_{7}$ & $21$ & $y_{8}$ &  $18$ & $z_{3}$    & $6$ & $z_{2}$ & $7$ & $z_1$ \\
\hline
III.3 & $y_{1}$ &  $1$ & $y_{2}$& $16$ & $z_{4}$ & $15$ & $y_{9}$ & $20$ & $y_{7}$ & $21$ & $y_{8}$ &  $18$ & $z_{3}$    & $6$ & $z_{2}$ & $7$ & $z_1$ \\
\hline
III.4 & $y_{1}$ &  $1$ & $y_{2}$& $16$ & $y_{9}$ & $23$ & $z_{4}$ & $20$ & $y_{7}$ & $21$ & $y_{8}$ &  $18$ & $z_{3}$    & $6$ & $z_{2}$ & $7$ & $z_1$ \\
\hline
III.5 & $y_{1}$ &  $1$ & $y_{2}$& $16$ & $y_{9}$ & $23$ & $y_{7}$ & $24$ & $z_{4}$ & $21$ & $y_{8}$ &  $18$ & $z_{3}$    & $6$ & $z_{2}$ & $7$ & $z_1$ \\
\hline
III.6 & $y_{1}$ &  $1$ & $y_{2}$& $16$ & $y_{9}$ & $23$ & $y_{7}$ & $24$ & $y_{8}$ & $19$ & $z_{4}$ &  $18$ & $z_{3}$    & $6$ & $z_{2}$ & $7$ & $z_1$ \\
\hline
\end{tabular}
\caption{Область III: порядок 3-камер и особых точек.}
\label{Tab:order_region_III}
\end{table}

\begin{table}
\begin{tabular}[t]{|c||c|c|c|c|c|c|c|c|c|c|c|c|c|c|c|c|c|c|c|c|c|c|c|c|}
\hline
II.1 & $y_{1}$ &  $1$ &   $z_{4}$ & 13 &$y_{2}$& $15$ & $y_{11}$ & $25$ & $y_{7}$ & $26$ & $z_{3}$ & $27$ & $y_{10}$ & $22$ & $y_{8}$  \\
\hline
II.3 & $y_{1}$ &  $1$ &  $z_{4}$ & $13$ &$y_{2}$& $15$ & $y_{11}$ & $25$ & $y_{7}$ & $26$ & $y_{10}$ & $21$ & $z_{3}$ & $22$ & $y_{8}$  \\
\hline
II.4 & $y_{1}$ &  $1$ & $z_{4}$ & $13$ &$y_{2}$& $15$ & $y_{11}$ & $25$ & $y_{10}$ & $20$ & $y_{7}$ & $21$ & $y_{8}$ & $18$ & $z_{3}$   \\
\hline
II.5 & $y_{1}$ &  $1$ & $y_{2}$& $16$ & $z_{4}$ & $15$ & $y_{11}$ & $25$ & $y_{10}$ & $20$ & $y_{7}$ & $21$ & $y_{8}$ & $18$ & $z_{3}$  \\
\hline
II.6 & $y_{1}$ &  $1$ & $y_{2}$& $16$ & $y_{11}$ & $28$ & $z_{4}$ & $25$ & $y_{10}$ & $20$ & $y_{7}$ & $21$ & $y_{8}$ & $18$ & $z_{3}$  \\
\hline
II.7 & $y_{1}$ &  $1$ & $y_{2}$& $16$ & $y_{11}$ & $28$ & $y_{10}$ & $23$ & $z_{4}$ & $20$ & $y_{7}$ & $21$ &  $y_{8}$ & $18$ & $z_{3}$  \\
\hline
II.8 & $y_{1}$ &  $1$ & $y_{2}$& $16$ & $y_{11}$ & $28$ & $y_{10}$ & $23$ & $y_{7}$ & $24$ & $z_{4}$ & $21$ & $y_{8}$ & $18$ & $z_{3}$  \\
\hline
II.9 & $y_{1}$ &  $1$ & $y_{2}$& $16$ & $y_{11}$ & $28$ & $y_{10}$ & $23$ & $y_{7}$ & $24$ & $y_{8}$ & $19$ & $z_{4}$ & $18$ & $z_{3}$  \\
\hline
\hline
II.2 & $y_{1}$ &  $1$ & $z_{4}$ & $13$ &$y_{2}$& $15$ & $y_{11}$ & $25$ & $y_{10}$ & $20$ & $y_{7}$ & $21$ & $z_{3}$ & $22$ & $y_{8}$  \\
\hline
II.10& $y_{1}$ &  $1$ & $z_{4}$ & $13$ &$y_{2}$& $15$ & $y_{11}$ & $25$ & $y_{7}$ & $26$ & $y_{10}$ & $21$ & $y_{8}$ & $18$ & $z_{3}$  \\
\hline
II.11& $y_{1}$ &  $1$ & $y_{2}$& $16$ & $z_{4}$ & $15$ & $y_{11}$ & $25$ & $y_{7}$ & $26$ & $y_{10}$ & $21$ & $y_{8}$ & $18$ & $z_{3}$  \\
\hline
II.12& $y_{1}$ &  $1$ & $y_{2}$& $16$ & $y_{11}$ & $28$ & $z_{4}$ & $25$ & $y_{7}$ & $26$ & $y_{10}$ & $21$ & $y_{8}$ & $18$ & $z_{3}$  \\
\hline
II.13& $y_{1}$ &  $1$ & $y_{2}$& $16$ & $y_{11}$ & $28$ & $y_{7}$ & $29$ & $z_{4}$ & $26$ & $y_{10}$ & $21$ & $y_{8}$ & $18$ & $z_{3}$  \\
\hline
II.14& $y_{1}$ &  $1$ & $y_{2}$& $16$ & $y_{11}$ & $28$ & $y_{7}$ & $29$ & $y_{10}$ & $24$ & $z_{4}$ & $21$ & $y_{8}$ & $18$ & $z_{3}$  \\
\hline
II.15& $y_{1}$ &  $1$ & $y_{2}$& $16$ & $y_{11}$ & $28$ & $y_{7}$ & $29$ & $y_{10}$ & $24$ & $y_{8}$ & $19$ & $z_{4}$ & $18$ & $z_{3}$  \\
\hline
\end{tabular}
\caption{Область II: порядок 3-камер и особых точек. Нетривиальная часть.}
\label{Tab:order_region_II}
\end{table}

\begin{table}
\begin{tabular}{|c||c|c|c|c|c|c|c|c|c|c|c|c|c|c|c|c|c|c|c|c|c|c|c|c|}
\hline
I.1 &$y_{1}$ &  $1$ & $z_{4}$ & $13$ & $y_{12}$ & $30$ & $z_3$ & $31$ & $y_{13}$ & $27$ & $y_{10}$ & $22$ & $y_{8}$ & $6$ & $z_{2}$ & $7$ & $z_1$\\
\hline
I.2 &$y_{1}$ &  $1$ & $z_{4}$ & $13$ & $y_{12}$ & $30$ &$y_{13}$& $26$ & $z_{3}$ & $27$ & $y_{10}$ & $22$ & $y_{8}$ & $6$ & $z_{2}$ & $7$ & $z_1$\\
\hline
I.3 &$y_{1}$ &  $1$ & $z_{4}$ & $13$ & $y_{12}$ & $30$ &$y_{13}$& $26$ & $y_{10}$ & $21$ & $z_{3}$ & $22$ & $y_{8}$  & $6$ & $z_{2}$ & $7$ & $z_1$\\
\hline
I.4 &$y_{1}$ &  $1$ & $z_{4}$ & $13$ & $y_{12}$ & $30$ &$y_{13}$& $26$ & $y_{10}$ & $21$ & $y_{8}$ & $18$ & $z_{3}$ & $6$ & $z_{2}$ & $7$ & $z_1$\\
\hline
I.9 &$y_{1}$ &  $1$ & $y_{12}$ & $32$ & $z_{4}$ & $30$ & $z_{3}$ & $31$ &$y_{13}$& $27$ & $y_{10}$ & $22$ & $y_{8}$ & $6$ & $z_{2}$ & $7$ & $z_1$\\
\hline
I.10&$y_{1}$ &  $1$ & $y_{12}$ & $32$ & $z_{4}$ & $30$ &$y_{13}$& $26$ & $z_{3}$ & $27$ & $y_{10}$ & $22$ & $y_{8}$ & $6$ & $z_{2}$ & $7$ & $z_1$\\
\hline
I.11&$y_{1}$ &  $1$ & $y_{12}$ & $32$ & $z_{4}$ & $30$ &$y_{13}$& $26$ & $y_{10}$ & $21$ & $z_{3}$ & $22$ & $y_{8}$ & $6$ & $z_{2}$ & $7$ & $z_1$\\
\hline
I.5 &$y_{1}$ &  $1$ & $y_{12}$ & $32$ & $z_{4}$ & $30$ &$y_{13}$& $26$ & $y_{10}$ & $21$ & $y_{8}$ & $18$ & $z_{3}$ & $6$ & $z_{2}$ & $7$ & $z_1$\\
\hline
I.6 &$y_{1}$ &  $1$ & $y_{12}$ & $32$ &$y_{13}$& $29$ & $z_{4}$ & $26$ & $y_{10}$ & $21$ & $y_{8}$ & $18$ & $z_{3}$ & $6$ & $z_{2}$ & $7$ & $z_1$\\
\hline
I.7 &$y_{1}$ &  $1$ & $y_{12}$ & $32$ &$y_{13}$& $29$ & $y_{10}$ & $24$ & $z_{4}$ & $21$ & $y_{8}$ & $18$ & $z_{3}$ & $6$ & $z_{2}$ & $7$ & $z_1$\\
\hline
I.8 &$y_{1}$ &  $1$ & $y_{12}$ & $32$ &$y_{13}$& $29$ & $y_{10}$ & $24$ & $y_{8}$ & $19$ & $z_{4}$ & $18$ & $z_{3}$ & $6$ & $z_{2}$ & $7$ & $z_1$\\
\hline
\end{tabular}
\caption{Область I: порядок 3-камер и особых точек.}
\label{Tab:order_region_I}
\end{table}

\begin{table}
\centering
\begin{tabular}[t]{|c||c|c|c|c|c|c|c|c|c|c|c|c|c|}
\hline
$\varkappa = 0$ & $H$ & $U_1$ & $U_2$ & $U_3$ & $M_1$ & $M_2$ & $e_1$ & $e_2$ & $c_1$ & $c_2$ & $h_1$ & $h_2$  \\
\hline
 $\varkappa > 0$ & $y_{1}$ & $y_{3}$ & $y_{7}$ & $y_{12}$ & $y_{10}$ & $y_{11}$ & $y_{2}$ & $y_{13}$ & $y_{6}$ & $y_{9}$ & $y_{5}$ & $y_{8}$ \\
\hline
\end{tabular}
\caption{Соответствие обозначений для  семейств особых точек.}
\label{Tab:denotions_singular_points}
\end{table}

\begin{table}
\centering
\begin{tabular}{c}
\includegraphics[width=0.9\linewidth]{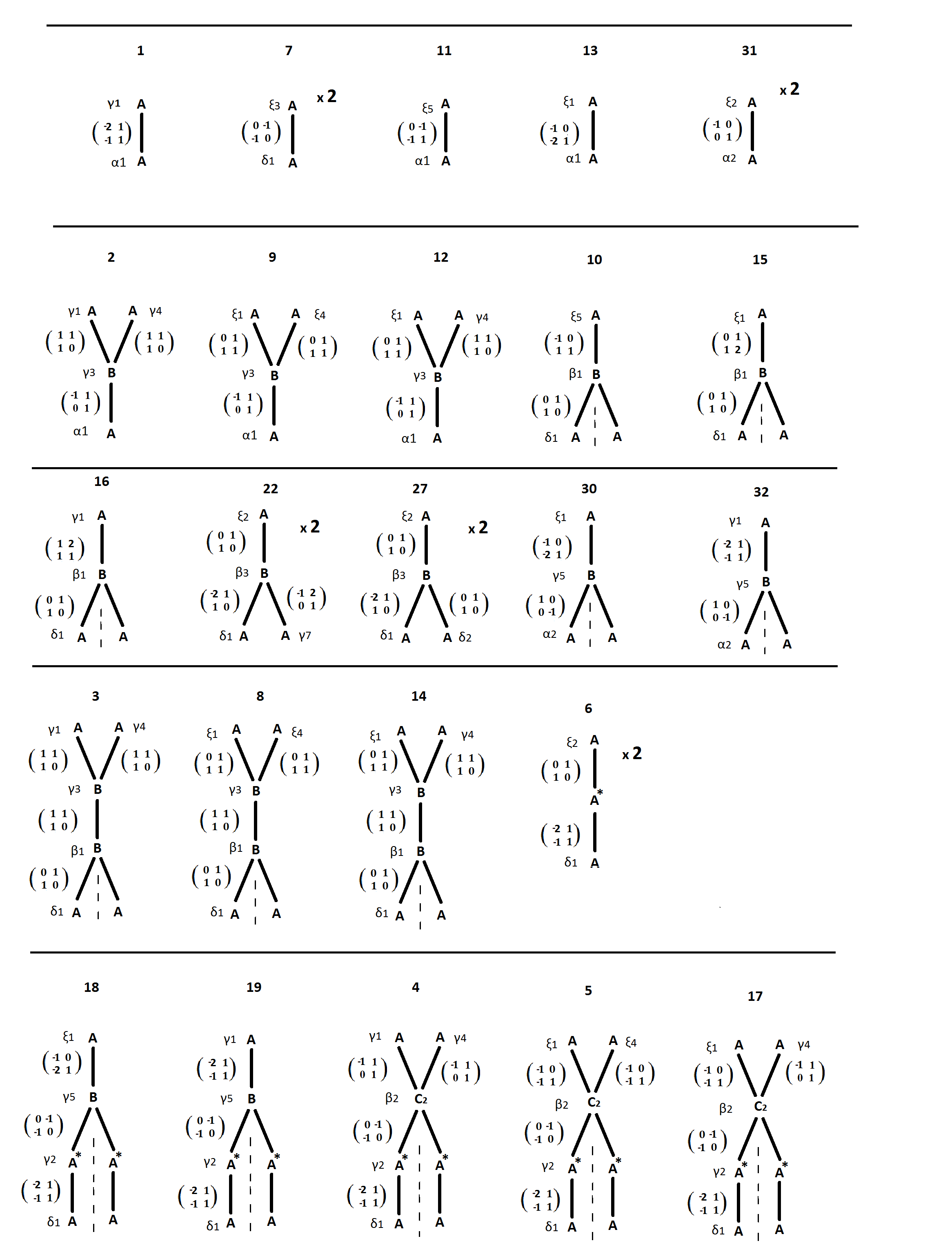}
\end{tabular}
   \caption{Список молекул с матрицами склейки. Часть 1.} \label{Tab:Molecules_matrix_1}
\end{table}

\begin{table}
\centering
\begin{tabular}{c}
\includegraphics[width=0.95\linewidth]{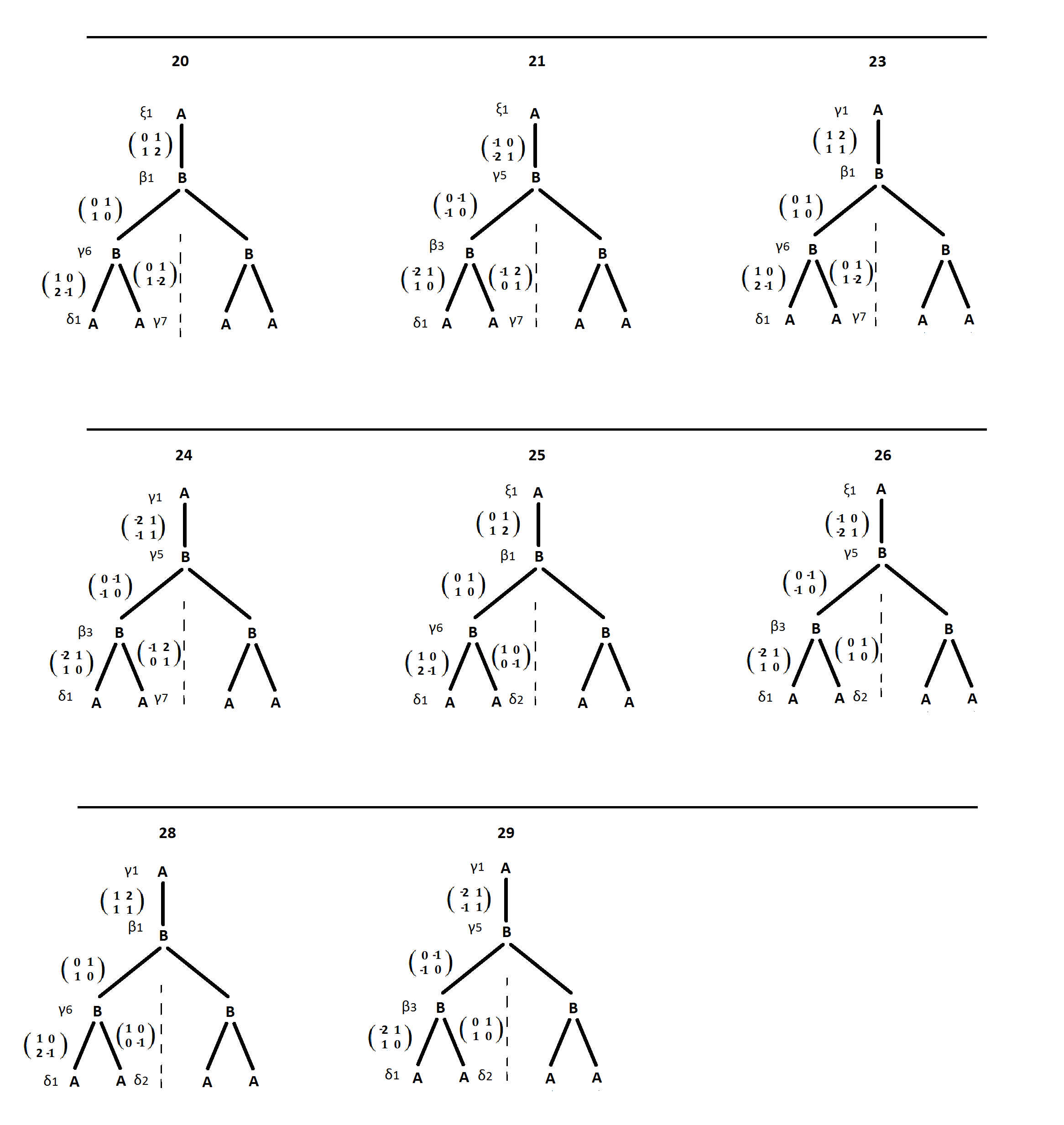}
\end{tabular}
   \caption{Список молекул с матрицами склейки. Часть 2.} \label{Tab:Molecules_matrix_2}
\end{table}

\begin{table}
\begin{tabular}[t]{|c||c|c|c|c|c|c|c|c|c|c|c|c|c|c|c|c|c|c|c|c|c|c|c|c|c|c|c|c|c|c|c|c|}
\hline
грань & $p_{1}$ & $p_{2}$ & $p_{3}$ & $p_{4}$ & $p_{5}$ & $p_{6}$  & $p_{7}$ & $p_{8}$ & $p_{9}$ & $p_{10}$ & $p_{11}$ & $p_{12}$ & $p_{13}$ & $p_{14}$   \\
\hline
точка & $y_{1}$ & $y_{2}$ & $y_{3}$ & $z_{5}$ & $z_{2}$ & $z_{1}$ & $y_{3}$ & $y_{2}$ & $z_{8}$ & $z_{10}$ & $z_{9}$ &$y_{2}$ & $z_{9}$  &  $y_{6}$ \\
\hline
выше  & 1 & 3 & 5 & 6 & 7 & $\oslash$ & 4 & 8 & 10 & 7 & 11 & 10 & $\oslash$ & 2  \\
\hline
ниже  & $\oslash$ & 2 & 8 & 5 & 6 & 7 & 3 & 9 & 8 & 10 & 9 & 11 & 11 & 1  \\
\hline
\hline
грань & $p_{15}$ & $p_{16}$ & $p_{17}$ & $p_{18}$ & $p_{19}$ &$p_{20}$  &$p_{21}$ &$p_{22}$ &$p_{23}$ &$p_{24}$ &$p_{25}$ &$p_{26}$ &$p_{27}$ &$p_{28}$ \\
\hline
точка & $z_{4}$ & $z_{6}$& $z_{4}$& $y_{6}$& $z_{11}$& $y_{2}$& $z_{6}$ & $y_{2}$ & $y_{6}$ & $z_{11}$ & $y_{3}$ & $z_{6}$ & $y_{2}$ & $z_{4}$ \\
\hline
выше  & 12 & 9 & 13 & 12 & 11 & 14 & 8  & 15 & 14 & 10 & 17 & 5 & 16 & 15 \\
\hline
ниже  & 2 & 12 & 1 & 13 & 13 & 12 & 14 & 13 & 15 & 15 & 14 & 17 & 1 & 16 \\
\hline
\hline
грань & $p_{29}$ & $p_{30}$ & $p_{31}$ & $p_{32}$ & $p_{33}$ &$p_{34}$  &$p_{35}$ &$p_{36}$ &$p_{37}$ &$p_{38}$ &$p_{39}$ &$p_{40}$ &$p_{41}$ &$p_{42}$ \\
\hline
точка & $y_{6}$ & $z_{4}$ & $z_{4}$& $y_{5}$& $z_{3}$& $y_{5}$& $z_{4}$ & $y_{9}$ & $y_{7}$ & $z_{3}$ & $y_{8}$ & $y_{8}$ & $y_{9}$ & $z_{4}$ \\
\hline
выше  & 3 & 17 & 14 & 18 & 6 & 19 & 18 & 20 & 21 & 22 & 6 & 18 & 23 & 20 \\
\hline
ниже  & 16 & 4 & 3 & 17 & 18 & 4 & 19 & 15 & 20 & 21 & 22 & 21 & 16 & 23 \\
\hline
\hline
грань & $p_{43}$ & $p_{44}$ & $p_{45}$ & $p_{46}$ & $p_{47}$ &$p_{48}$  &$p_{49}$ &$p_{50}$ &$p_{51}$ &$p_{52}$ &$p_{53}$ &$p_{54}$ &$p_{55}$ &$p_{56}$ \\
\hline
точка & $y_{7}$ & $z_{4}$ & $y_{8}$& $y_{11}$& $y_{7}$& $z_{3}$& $y_{10}$ & $y_{10}$ & $y_{10}$ & $y_{11}$ & $z_{4}$ & $y_{10}$ & $y_{7}$ & $z_{4}$ \\
\hline
выше  & 24 & 21 & 19 & 25 & 26 & 27 & 22 & 21 & 20 & 28 & 25 & 23 & 29 & 26 \\
\hline
ниже  & 23 & 24 & 24 & 15 & 25 & 26 & 27 & 26 & 25 & 16 & 28 & 28 & 28 & 29 \\
\hline
\hline
грань & $p_{57}$ & $p_{58}$ & $p_{59}$ & $p_{60}$ & $p_{61}$ &$p_{62}$ & $p_{63}$ & $p_{64}$ & $p_{65}$ & $p_{66}$ & & & & \\
\hline
точка & $y_{10}$ & $y_{12}$ & $z_{3}$ & $y_{13}$ & $y_{13}$ & $y_{12}$ &  $z_{4}$ & $y_{13}$ & $S^1$ & $S^2$ & & & & \\
\hline
выше  & 24 & 30 & 31 & 27 & 26 & 32 & 30 & 29 & $\oslash$ & $\oslash$ &  &  &  &  \\
\hline
ниже  & 29 & 13 & 30 & 31 & 30 & 1 & 32 & 32 & 13 & 31 &  &  &  &  \\
\hline
\hline
\end{tabular}
\caption{Перестройки меченых молекул для компактного случая Ковалевской.}
\label{Tab:polyhedron_facets}
\end{table}

\begin{table}
\centering
\begin{tabular}[t]{|c|c|c|c|c|c|c|c|c|c|c|c|c|c|}
\hline
 $-l$ & $+l$ & $-r$ & $+r$ & $lt$ & $rt$ & $cusp$ & $l$ & $int$ \\
\hline
 $y_{1}$ & $y_{3}, y_{7}, y_{12} z_{9}, z_{10}$ & $z_{4}$ & $z_{3}, z_{6}, z_{11}$ & $y_{5}, y_8$ & $z_2, z_{8}$ & $y_{6}, y_{9}$ &$y_{2}, y_{13}$ & $z_{5}$ \\
\hline
\end{tabular}
\caption{Семейства особых точек $\Sigma^{a, b}$  и поверхности первой серии}
\label{Tab:surface_family}
\end{table}

\begin{table}
\centering
\begin{tabular}[t]{|c||l|c|l|}
\hline
\hline
№ & формула  &$(x_1, x_2)$ & пары особых точек \\
\hline
\hline
  0 &$u(v) = 0$ & $(0, \, \infty)$ & \\
\hline
\hline
  1 &$v(u) = u$ & $(0, \, \infty)$ & $(-l, +l)_0$, $(-r, +r)_0$, $(l, lt, rt, z_1)_0$ \\
\hline
\hline
  2 &$v(u) = u +4\tau \sqrt{u} +4\tau^2$ & $(0, \, \infty)$ & $(+l, -r)_1$\\
\hline
\hline
  3 &$v(u) =  u - 4\tau \sqrt{u} +4\tau^2$ & $(0, \, \tau^2)$ & $(+l, +r)_1$\\
\hline
\hline
  4 & $v(u) = \cfrac{u^2}{\tau^2}$ & $(\tau^2, \, \infty)$ &  $(+l, lt)_0$, $(cusp, +l)_0$, $(cusp, lt)_0$\\
\hline
\hline
  5 & $v(u) = \tau \sqrt{u}$ & $(0, \, \tau^2)$ & $(+r, rt)_0$, $(cusp, +r)_0$, $(cusp, rt)_0$ \\
\hline
\hline
  7 & $v(u) = \cfrac{2u^{3/2}}{\sqrt{u}-\tau}$ & $(\tau^2, \, \infty)$ & $(lt, -r)_1$\\
\hline
\hline
  8 & $v(u) = \cfrac{2u^{3/2}}{\sqrt{u}+\tau}$ & $(\tau^2, \, \infty)$ & $(+r, lt)_1$\\
\hline
\hline
  9 & $u(v) = \tau^2$ & $\tau^2, \, \infty)$ & $(lt, +r)_1$, $(+r, int)_0$, $(lt, int)_1$ \\
\hline
\hline
  10 &$v(u) =  \tau^2$ & $(0, \, \tau^2)$ & $(+l, rt)_1$, $(+l, int)_0$, $(int, rt)_1$\\
\hline
\hline
  13 & $v(u) = 8\tau\sqrt{u}$ & $(0, \, 64\tau^2)$ & $(cusp, -r)_1$\\
\hline
\hline
  14 & $v(u) = \cfrac{1}{4}\left(\tau + \cfrac{u}{\tau}\right)^2$ & $(0, \, \tau^2)$ & $(l, +l)_1$ \\
\hline
\hline
 14 & $v(u) = \cfrac{1}{4}\left(\tau + \cfrac{u}{\tau}\right)^2$ & $(\tau^2, \, \infty)$ & $(l, +l)_1, (+l, rootl)_1 (l, rootl)_1$\\
\hline
\hline
  15 & $v(u) = 2u+2\tau \sqrt{u} +\tau^2$ & $(0, \, \infty)$ & $(l, -r)_1$\\
\hline
\hline
  16 & $v(u) = 2u- 2\tau \sqrt{u} +\tau^2$ & $(0, \, \infty)$ & $(l, +r)_1$\\
\hline
\hline
  17 & $v(u) =  \cfrac{(2u + \tau^2)^3}{27 \tau^2 u} $ & $(0, \, \infty)$ & $(cusp, l)_1$\\
\hline
\hline
  18 & $v(u) = \cfrac{\tau^2}{2}$ & $\left(0 ,\cfrac{\tau^2}{2}\right)$ & $(l, rt)_1$\\
\hline
\hline
\end{tabular}
\caption{Разделяющие кривые для особых точек первой серии.}
\label{Tab:sep_curves_2.1}
\end{table}

\begin{table}
\centering
\begin{tabular}[t]{|c||l|c|l|}
\hline
\hline
№ & формула  &$(x_1, x_2)$ & пары особых точек \\
\hline
\hline
  6 & $u(v) = \cfrac{2v^{3/2}}{\sqrt{v}+\tau}$ & $(0, \, \tau^2)$ & $(rt, +l)_1$\\
\hline
\hline
  11 & $v(u) = \cfrac{\tau^2}{2} \left(1 + \cfrac{\tau^2 /2}{u - \tau^2 /2}\right)$ & $\left(\cfrac{\tau^2}{2}, \tau^2\right)$ & $(lt, rt)_1$\\
\hline
\hline
  12 & $v(u) =  \cfrac{u^2}{64\tau^2}$ & $(64 \tau^2, \, \infty)$ & $(-l, cusp)_1$\\
\hline
\hline
  19 & ${\bf a(b)} = 3 b^{2/3} c_1^{2/3} \varkappa -\varkappa^2 c_1^2 $ & $\oslash$ & (cusp, int)\\
\hline
\hline
\end{tabular}
\caption{Кривые с конечным числом точек разделяющего множества $\Theta$.}
\label{Tab:empty_curves}
\end{table}

\clearpage

\begin{table}
\centering
\begin{tabular}[t]{||c|c|c|c|c|c|c|c|}
\hline
\hline
 номер & название & вид в $(u, v)$ & вид в $(q, s)$  \\
\hline
$f_1$ & $f_l$ & $v = u$ & $ s(q) = q^2$  \\
\hline
$f_4$ & $f_r$ & $v = \cfrac{u^2}{\tau^2}$ & $ s(q) = \cfrac{q^{4/3}}{\tau^{2/3}}$ \\
\hline
$f_{9}$ & $f_m$ & $u = \tau^2$ & $ s(q) = \cfrac{q}{\tau}$ \\
\hline
$f_{14}$ & $f_t$ & $v = \cfrac{1}{4}\left(\tau + \cfrac{u}{\tau}\right)^2$ & $ q(s) = \tau s \sqrt{-1 + \cfrac{2}{\tau \sqrt{s}}\mathstrut}$
\\
\hline
$f_{24}$ & $f_k$ & $\oslash$ & $ s(q) = - \cfrac{1}{\tau^2} (2-3\tau^{2/3}q^{2/3}) 
+ \cfrac{2}{\tau^2} \left(1 - \tau^{2/3}q^{2/3}\right)^{\frac{3}{2}}$
\\
\hline
\hline
\end{tabular}
\caption{Вид кривых в координатах  $(q, s)$.}
\label{Tab:sep_curves:q_s-type}
\end{table}

\end{document}